\documentclass[11pt]{article}
\usepackage{amssymb,amsmath,amsfonts,amsthm}

\usepackage{appendix}
\usepackage{microtype}
\usepackage{bm}
\usepackage{enumitem}

\usepackage{tikz}
\usetikzlibrary{fit,calc,decorations.pathmorphing,positioning}

\usepackage[colorlinks, allcolors=blue]{hyperref}

\usepackage[margin=1in]{geometry}

\numberwithin{equation}{section}

\newtheorem{theorem}{Theorem}[section]

\newtheorem{lemma}[theorem]{Lemma}
\newtheorem{claim}[theorem]{Claim}
\newtheorem{definition}[theorem]{Definition}
\newtheorem{corollary}[theorem]{Corollary}

\newtheorem{remark}[theorem]{Remark}

\def\E{\mathbb{E}}
\def\P{\mathbb{P}}
\newcommand{\sA}{{\mathcal{A}}}
\newcommand{\sB}{{\mathcal{B}}}
\newcommand{\sF}{{\mathcal{F}}}
\newcommand{\sL}{{\mathcal{L}}}
\newcommand{\sP}{{\mathcal{P}}}
\newcommand{\R}  {\mathbb{R}}
\newcommand{\N}  {\mathbb{N}}

\newcommand{\ttl}{On Solutions to Graphon McKean--Vlasov SDEs  of Nemytskii-type}

\begin{document}

\title{\ttl}

\author{Sebastian Grube\footnote{Faculty of Mathematics, Bielefeld University, 33615 Bielefeld, Germany. E-mail: sgrube@math.uni-bielefeld.de
}
\and Guodong Pang\footnote{Computational Applied Mathematics and Operations Research, George R. Brown School of Engineering and Computing, Rice University, Houston, TX 77005. E-mail: gdpang@rice.edu
}
\and Michael R\"ockner\footnote{Faculty of Mathematics, Bielefeld University, 33615 Bielefeld, Germany. E-mail: roeckner@math.uni-bielefeld.de
}
\footnote{Academy of Mathematics and System Sciences, CAS, Beijing, China.
}
\footnote{School of Data Science, The Chinese University of Hong Kong, Shenzhen, China.
}
}
    
\date{\today}

\maketitle

\begin{abstract}
	We study an uncountable system of McKean--Vlasov SDEs with coefficients of Nemytskii-type which are driven by a family of essentially pairwise independent Wiener processes. These SDEs interact through a Graphon kernel by means of their one-dimensional time marginal law densities evaluated in the spatial coordinate.
	We prove the existence and uniqueness of probabilistically weak solutions to such SDEs under mild conditions on the drift coefficient and the Graphon kernel.
	Furthermore, we prove that these solutions are, in fact, probabilistically strong by essentially proving a (restricted) Yamada--Watanabe theorem on Fubini extension spaces.
	
	For the associated system of nonlinear Fokker--Planck equations, we prove a new uniqueness result where we can allow for density dependent diffusion coefficients of Nemytskii-type.
\end{abstract}



\allowdisplaybreaks
\section{Introduction}\label{section.introduction}
\subsection{Motivation}
In this work, we investigate the following system of McKean--Vlasov SDEs (abbreviated as MVSDE): For each $u \in [0,1]$
\begin{align}
& d X^u_t   = b\Bigg(t,X^u_t, \frac{1}{\int_0^1 K(u,v)dv} \int_0^1 K(u,v)  p^v_t(X^u_t) dv  \Bigg) dt  \notag\\
 &\qquad \qquad  + \sigma\Bigg(t,X^u_t, \frac{1}{\int_0^1 K(u,v)dv} \int_0^1 K(u,v)  p^v_t(X^u_t )  dv  \Bigg) d W^u_t\quad t\in [0,T], \label{GMVSDE.1}\\
&\sL_{X^u(t)}(dx)=p^u_t(x)dx,\label{GMVSDE.2}
\end{align}
where $T\in (0,\infty)$, $(W^u(t))_{t\in[0,T]}$, $u \in [0,1]$, is a family of essentially pairwise independent (short: e.p.i.; see Definition \ref{appendix.fubiniextension.definition.epi}) $d$-dimensional Wiener processes, $K$ is a given graphon, i.e.,\\$K:[0,1]\times[0,1]\to[0,1]$ is a symmetric and Borel measurable function, and $b$, $\sigma$ are given componentwise by Borel measurable functions
\begin{align*}
	b^i,\sigma^{ij}:[0,T]\times \R^d\times\R\to \R,\quad i,j\in \{1,\dots,d\}.
\end{align*}
Note that, for fixed $u \in [0,1]$, $X^u$ may itself solve an SDE whose coefficients depend on its own one-dimensional time marginal law densities $(p^u_\cdot)$, see Theorem~\ref{theorem.GFPE.stepKernel}. Hence, \eqref{GMVSDE.1}-\eqref{GMVSDE.2} is, in general, a family of McKean--Vlasov SDEs in the usual sense, indexed by the parameter $u \in [0,1]$.

Based on \cite{oelschlaeger1985law} and \cite{coppini2025nonlinear} (see also \cite{djete2026nonexchangeablemeanfieldgames} for a related framework), we expect that a solution $X$ to \eqref{GMVSDE.1}-\eqref{GMVSDE.2} is (at least heuristically) obtained as a limit of the solutions
to the heterogeneous \textit{moderately} interacting particle system
\begin{align}\label{GMVSDE.particleSystem}
    d X^{i,N}_t = b \left(t, X^{i,N}_t, \langle \mu^{i,N}_t, V^N(X^{i,N}_t - \cdot)\rangle\right) dt + \sigma \left(t, X^{i,N}_t, \langle \mu^{i,N}_t,V^N(X^{i,N}_t - \cdot)\rangle \right) d W^{i/N}_t,
    \end{align}
   where $i \in \{1,\dots,N\}$, $N\in \N$, and
\begin{align*}
		\mu^{i,N}_t:=\frac 1 {N_i}\sum \xi_{ij}^N\delta_{X^{j,N}(t)},
		\quad N_i:= \sum_{j=1}^N \xi_{ij}^N,
		\quad  \langle \mu^{i,N}_t,V^N(x - \cdot)\rangle = \frac{1}{N_i} \sum_{j=1}^N \xi^N_{ij} V^N(x-  X^{j,N}_t),
	\end{align*}
	 $V^N(x) := \frac{1}{\varepsilon^d_N} V^0\left(\frac{1}{\varepsilon_N} x\right)$ for some smooth probability density $V^0$ on $\R^d$, and $ \varepsilon_N = N^{-\beta/d}$ for some $\beta \in (0,1)$, i.e., exactly the range for $\beta$ for the \textit{moderate} interaction regime, see \cite{oelschlaeger1985law}.
	 Here, $(\xi_{ij}^N)$ are, e.g., mutually independent Bernoulli random variables which are determined by a graph, or equivalently, in the graphon framework, a step graphon $K^N:[0,1]\times [0,1]\to [0,1]$. The family $(K^N)$ is chosen such that $K^N\to K$ in a suitable sense (e.g., with respect to the cut-norm). The values of $K^N$ represent the connectivity probabilities between vertices $i$ and $j$ in a graph consisting of $N$ nodes. 
	More precisely, $K^N$ satisfies
	\begin{align*}
		K^N(u,v) = K^N(\lceil Nu\rceil/N,\lceil Nv\rceil/N)\quad \forall u,v\in [0,1],
	\end{align*}
	and the random variables $\xi_{ij}^N$ are constructed such that
	\begin{align*}
		K^N(i/N,j/N)=\P(\xi_{ij}^N=1),\quad \forall i,j \in \{1,\dots,N\},
	\end{align*}
	where $\P$ is the underlying probability measure.
	The dynamics on the nodes $i$ and $j$ is described by $X^{i,N}$ and $X^{j,N}$, respectively, and the quantity $N_i$, which was defined above, represents the degree of node $i$.
In this work, we focus solely on the weak and strong well-posedness of \eqref{GMVSDE.1}-\eqref{GMVSDE.2}.
The corresponding particle approximation and its convergence analysis will be the subject of future work.

Equations of type \eqref{GMVSDE.1}-\eqref{GMVSDE.2} can be used to model heterogeneous and local interactions between a continuum of agents.  Heterogeneous interactions play a fundamental role in numerous applications, including biology \cite{athreya2021graphon,sumi2026graphon},
neuroscience \cite{agathe2022multivariate,agathe2023long,jabin2024dense,jabin2026meanfield,forbes2025coherent} 
 epidemiology \cite{delmas2024individual,pang2026spatially,pang2025stochastic,delmas2025equivalence}, economics \cite{parise2023graphon,rokade2026graphon},  social sciences \cite{prisant2025opinion,baldassarri2024opinion}, and related disciplines.

\vspace{1em}
Interest in interacting particle systems on graphs and their scaling limits, especially in connection with mean field games, has grown significantly over the past decade, see, e.g., \cite{aurell2022stochastic,lacker2023label-state, bayraktar2023graphon,bayraktar2023propagation, coppini2025nonlinear, bayraktar2023graphonUniform,jabin2025meanField,crescenzo2026meanfield,crescenzo2026linear-quadratic}, and see \cite{coppini2022note} for the study of the corresponding Fokker--Planck equations. In all the mentioned work, only particle systems of weak interaction type (in the sense of Oelschl\"ager \cite{oelschlaeger1984martingale}) are studied.  Here, we particularly refer to the fundamental work \cite{bayraktar2023graphon}, which studies interacting particle systems and their convergence to a corresponding limiting equation in the case where the coefficients of the stochastic differential equations constituting the particle system depend linearly on the distribution variable. A closely related framework in the nonlinear setting is studied in \cite{coppini2025nonlinear}. For comparison with our results, we describe this framework in more detail below. In \cite{coppini2025nonlinear}, the authors study the following mean field interacting particle system
	\begin{align}\label{GMVSDE.weakly.particleSystem}
		dX^{i,N}(t)=\bar{b}\left(t,X^{i,N}(t),\mu^{i,N}_t\right)dt +\bar{\sigma}\left(t,X^{i,N}(t),\mu^{i,N}_t\right)dW^{i/N}(t),\ \ t\in [0,T],
	\end{align}
	where $i=1,\dots,N$, $N \in \N$, and $W^u, u \in [0,1]$, is again a family of e.p.i. Wiener processes. Here, $\bar{b}^i,\bar{\sigma}^{ij}:[0,T]\times \mathcal{P}_2(\R^d)\to \R$ are given Borel measurable functions, where $\mathcal{P}_2(\R^d)$ denotes the set of all Borel probability measures on $\R^d$ with finite second moment and is considered together with the usual $2$-Wasserstein distance $\mathcal{W}_2(\cdot,\cdot)$.

Taking $N\to \infty$ in \eqref{GMVSDE.weakly.particleSystem}, the limiting dynamics of the particle system are described by the following system of equations with nonlocal dependence on the one-dimensional time marginal laws of the solution processes: For all $u\in [0,1]$,
\begin{align}
& d X^u_t   = \bar{b}\Bigg(X^u_t, \frac{1}{\int_0^1 K(u,v)dv} \int_0^1 K(u,v)  \mu^v_t(dy)  dv  \Bigg) dt  \label{GMVSDE.weakly.1}\\
 &\qquad \qquad  + \bar{\sigma}\Bigg(X^u_t, \frac{1}{\int_0^1 K(u,v)dv} \int_0^1 K(u,v)  \mu^v_t(dy)  dv  \Bigg) d W^u_t\,, \quad t\in [0,T], \notag\\
&\sL_{X^u(t)}(dy)=\mu^u_t(dy)\,.\label{GMVSDE.weakly.2}
\end{align}
	
Note that the coefficients in our equation \eqref{GMVSDE.1}-\eqref{GMVSDE.2} are intrinsically different from \eqref{GMVSDE.weakly.1}-\eqref{GMVSDE.weakly.2}.
To illustrate that, consider, without loss of generality, the drift coefficient of \eqref{GMVSDE} as a map in a $u$-constant distribution variable
\begin{align}\label{introduction.map.Nemytskii}
	\mu=\mu_a(x)dx + \mu_s(dx)\mapsto\ &b\left (t,x,\frac{1}{\int_0^1 K(u,v)dv} \int_{[0,1]} K(u,v)  \mu_a(x)dv\right) \notag\\
    &=b\left (t, x,\mu_a(x)\right),
\end{align}
for fixed $t\in [0,T], x\in \R^d$ and $u \in [0,1]$, where $\mu_s$ is the singular part of $\mu$ with respect to Lebesgue measure.
Here we always choose the Lebesgue version of $\mu_a$ by setting $\mu_a$ equal to zero on the complement of its Lebesgue points. 
As shown in \cite{grube2026strong}, see also \cite{ren2022linearization}, this map is typically discontinuous and merely Borel measurable. More precisely, one can  construct an abundance of continuous functions $b$ such that the map in \eqref{introduction.map.Nemytskii} is discontinuous with respect to usual topologies on the space of Borel probability measures on $\R^d$, like the narrow topology, the one induced by the Wasserstein distance, or bounded variation norm. 
Hence, now considering $b$ as a map in a general distribution variable, i.e.,
\begin{align}
	\mathcal{M}\ni(\mu^v)_{v\in [0,1]}&=(\mu^v_a(x)dx + \mu^v_s(dx))_{v\in [0,1]} \notag\\
    &\mapsto b\left (t,x,\frac{1}{\int_0^1 K(u,v)dv} \int_0^1 K(u,v)  \mu^v_a(x)dv\right),
\end{align}
$b$ is generically discontinuous,
where we consider
\begin{align*}
	\mathcal{M}:=\Big\{ (\mu^v)_{v\in [0,1]}\in \mathcal{P}_2(\R^d)^{[0,1]}:\ & v\mapsto \mu^v(B) \text{ is measurable }\forall B \in \sB(\R^d),\\
    &\sup_{v\in [0,1]}\int |x|^2 \mu^v(dx)<\infty\Big\}
\end{align*}
together with the uniform Wasserstein distance $\sup_{v\in [0,1]} \mathcal{W}_2(\cdot^v,\cdot^v)$ on $\mathcal{M}\times\mathcal{M}$.
This topology is essential in the study of \eqref{GMVSDE.weakly.1}-\eqref{GMVSDE.weakly.2} in \cite{coppini2025nonlinear}, see also \cite{bayraktar2023graphon}.

Note that \eqref{introduction.map.Nemytskii} provides one way to represent the singular Nemytskii-type dependence of the coefficients in \eqref{GMVSDE} on the distribution variable.

\subsection{Goals of this paper}\label{section.goals}
Instead of \eqref{GMVSDE.1}-\eqref{GMVSDE.2}, it is conceptually more convenient to study the following more general system of MVSDEs: 
{\small\begin{align}\label{GMVSDE}\tag{GMVSDE}
\begin{cases}
	&d X^u_t  = b\Big(t,X^u_t,  \langle\tilde{K}_1(u), p^\bullet_t(X^u_t)\rangle  \Big) dt   + \sigma\Big(t,X^u_t,  \langle\tilde{K}_2(u), p^\bullet_t(X^u_t)\rangle  \Big) d W^u_t,\\
&\sL_{X^u(t)}(dx)=p^u_t(x)dx,\quad t\in [0,T],\quad u \in [0,1].
\end{cases}
\end{align}}
where $\tilde{K}_i:[0,1]\mapsto \mathcal{M}_+^b([0,1])$, $i=1,2$, are given Borel measurable maps.
Here, $\mathcal{M}_{(+)}^b([0,1])$ denotes the set of all (nonnegative and) finite Borel measures on $[0,1]$.
On $\mathcal{M}_{(+)}^b([0,1])$ we consider the  usual total variation norm $\|\cdot\|_{\mathrm{TV}}$.
We set $\langle \mu,\varphi\rangle:=\int \varphi d\mu$, for all $\mu \in \mathcal{M}_+^b([0,1])$ and all bounded Borel measurable functions $\varphi: [0,1]\to\R$.
In the coefficients of \eqref{GMVSDE}, we consider $(u,x)\to  p^u_t(x)$ as the probability density of the measure $\sL_{X^u(t)}(dx)du$ (which makes sense in our framework, since the map $u\mapsto \sL_{X^u(t)}(dx)$ will be Lebesgue measurable).
Hence, for $i=1,2$, $\tilde{K}_i$ needs to respect $L^1([0,1])$-equivalence classes, which in our case is guaranteed by additionally requiring that, for all functions $\varphi,\bar{\varphi}:[0,1]\to \R$ in the same equivalence class in $L^1([0,1])$,
\begin{align*}
	\langle \tilde{K}_i(u),\varphi\rangle=\langle \tilde{K}_i(u),\bar{\varphi}\rangle\quad  \text{for $du$-a.e. $u \in [0,1]$}.
\end{align*}
In order to make rigorous sense of the coefficients in \eqref{GMVSDE}, we need to choose a Borel measurable version of the map
\begin{align}\label{introduction.goals:map.K}
[0,1]\times\R^d\ni (u,x)\to  \langle \tilde{K}_i(u),p^\bullet_t(x)\rangle,	
\end{align}
which we do as follows.
We assume that, for some $p\in [1,\infty)$, $L^p([0,1])\ni\varphi \mapsto (u\mapsto \langle\tilde{K}_i(u),\varphi\rangle)\in L^p([0,1])$ is a bounded operator, which we do in all our main results, see Section \ref{section.mainresults}.
Then, if $(u,x)\mapsto p^u_t(x) \in L^p_{\mathrm{loc}}([0,1]\times\R^d)$, the map in \eqref{introduction.goals:map.K} is a well-defined element in $L^p_{\mathrm{loc}}([0,1]\times\R^d)$.
Now, we always consider the Lebesgue version of \eqref{introduction.goals:map.K}, which is obtained by setting the function to be zero outside of its Lebesgue points.

A (probabilistically) weak solution to \eqref{GMVSDE} is a family of classical weak solutions \\$(X^u,W^u)_{u\in[0,1]}$,
and it is called a (probabilistically) strong solution if, for each $u \in [0,1]$, $X^u=F(X^u(0),u,W^u)$ for a functional $F:\R^d\times [0,1]\times C_0([0,T];\R^d)\to C([0,T];\R^d)$ satisfying certain measurability and adaptedness conditions.
For the precise notions of solutions to \eqref{GMVSDE} we refer to Section \ref{section.solutionConcepts.GMVSDE} below.

As is well known, there is a natural connection between SDEs and Fokker--Planck equations. This is also the case in our situation.
Assume that there is a weak solution to \eqref{GMVSDE}. Then, by It\^o's formula, the family of probability measures $(p^u_t(x)dx)_{t\in [0,T]}, u \in [0,1]$, as in \eqref{GMVSDE}, solves the following uncountable system of \textit{nonlinear Fokker--Planck equations} with $a^{ij}:=((1\slash  2)\sigma\sigma^*)^{ij}$: For $u \in [0,1]$,
\begin{align}
    \partial_t p^u_t(x)  = & -\mathrm{div}\left(b\left(t,x,  \langle\tilde{K}_1(u), p^\bullet_t(x)\rangle  \right)p_t^u(x)\right)\notag\\
    &+ \partial_{x_i}\partial_{x_j}\left(a_{ij}\left(t,x,  \langle\tilde{K}_2(u), p^\bullet_t(x)\rangle  \right)p_t^u(x)\right),\quad (t,x) \in [0,T]\times\R^d,
\label{GFPE}\tag{GFPE}
\end{align}
in the Schwartz distributional sense. The precise definition of a solution can be found in Section \ref{section.solutionConcepts.GFPE} below.

\vspace{1em}

Let us briefly summarize our main results. The precise statements can be found in Section \ref{section.mainresults}.
In this work, we prove the existence of a weak solution to \eqref{GMVSDE} whose initial condition is given by a uniformly bounded probability density, provided $\tilde{K}_1$ satisfies some mild conditions (in particular, it defines a bounded operator on $L^1([0,1]$)), $b$ is uniformly bounded, and merely uniformly H\"older continuous in its last component, and the diffusion coefficient $\sigma$ is constant, see Theorem~\ref{theorem.GMVSDE.weakExistence}.
Note that we do not impose any regularity assumption on the drift coefficient in its spatial component. Furthermore, the assumption of a constant diffusion coefficient is imposed to keep the length of this paper, as well as our technical difficulties, within reasonable margins. Our approach is based on an Euler scheme, where we adopt techniques from \cite{hao2021euler} and apply them to our setting. This adaption, however, entails a substantial amount of technical difficulties arising from the continuum of equations we need to solve simultaneously. 
We note that the multiplicative case with a nondegenerate diffusion coefficient can be treated by replacing the properties and estimates of the fundamental solution to the parabolic heat equation (see Appendix~\ref{appendix.heatKernel}), which are essential to our approach, with the corresponding heat kernel estimates from \cite{menozzi2021density}.
 Furthermore, we conclude that these weak solutions are even strong by essentially proving a (restricted) Yamada--Watanabe theorem on Fubini extensions, see (the proof of) Theorem~\ref{theorem.GMVSDE.strongExistence}. Here, we also refer to the articles \cite{grube2023thesis,grube2021strong}.
Moreover, we prove a uniqueness result for Schwartz distributional solutions to \eqref{GFPE} with non-constant diffusion term, where $b$ is assumed to be uniformly bounded and uniformly Lipschitz continuous in its last component, and $\tilde{K}_1$ is a bounded operator on $L^2([0,1])$. Here, we consider $\tilde{K}_2(x,dy)=\delta_x(dy)$, and the diffusion coefficient in \eqref{GFPE} is of the form
\begin{align*}
	a_{ij}\left(t,x,  \langle\tilde{K}_2(u), p^\bullet_t(x)\rangle  \right)
	=\delta_{ij}\sqrt{2\beta(p^u_t(x))/p^u_t(x)},
\end{align*}
where $\beta\in C^1(\R)$ is a given function such that $\beta(0)=0$, $\beta'\geq \gamma_0$, for some $\gamma_0>0$, see Theorem~\ref{theorem.GFPE.uniqueness}.
Moreover, we investigate the influence of the initial condition and the interaction kernel on the solution to \eqref{GFPE}, see Theorem~\ref{theorem.GFPE.stepKernel}.
Finally, in Section \ref{section.superpositionPrinciple}, we comment on the Ambrosio--Figalli--Trevisan superposition principle \cite{figalli2008martingale,trevisan2016well-posedness} (see also \cite{bogachev2021superposition} for a recent generalisation) relating solutions to \eqref{GFPE} to weak solutions to \eqref{GMVSDE}.

\subsection{Literature}\label{section.literature}
Both equations \eqref{GMVSDE} and \eqref{GFPE}, respectively, appear to be relatively new to the literature in this generality.
However, let us point out that when choosing $\tilde{K}_1(u,dv)=\tilde{K}_2(u,dv)=\delta_u(dv)$ and starting the respective equation with $\sL_{X^u_0}$ or $\nu_0^u$, respectively, being constant in the parameter $u$, the equations decouple such that \eqref{GMVSDE} essentially reduces to the following MVSDE with coefficients of Nemytskii-type
\begin{align}\label{GMVSDE.Nemytskii}
\begin{cases}
	&d X_t  = b\left(t,X_t,  p_t(X_t)  \right) dt   + \sigma\left(t,X_t,  p_t(X_t)  \right) d W_t,\quad t\in [0,T],\\
&\sL_{X(t)}(dx)=p_t(x)dx.
\end{cases}
\end{align}
The corresponding substitute for \eqref{GFPE} is the following nonlinear Fokker--Planck equation \begin{align}
    \partial_t p_t(x)  = -\mathrm{div}\left(b\left(t,x, p_t(x)  \right)p_t(x)\right)
    + \partial_{x_i}\partial_{x_j}\left(a_{ij}\left(t,x,  p_t(x)  \right)p_t^u(x)\right),\label{GFPE.Nemytskii}
\end{align}
where $(t,x) \in [0,T]\times\R^d$.
The existence and uniqueness of Schwartz distributional solutions to \eqref{GFPE.Nemytskii} and probabilistically weak solutions to \eqref{GMVSDE.Nemytskii} were investigated under very mild conditions on the coefficients (possibly degenerate diffusion coefficients) in \cite{barbu2018prob,barbu2020fromNonlinearFPE,barbu2021solutions,barbu2021weakUniqueness,barbu2023uniqueness}.
The long-time behavior of solutions to \eqref{GFPE.Nemytskii} was investigated in \cite{barbu2023equilibrium}.
Note that \eqref{GFPE.Nemytskii} includes the classical porous medium equation ($b\equiv 0$, $a_{ij}(t,x,r)\equiv \delta_{ij}|r|^{m-1}$, $m>1$).
The fundamental idea and approach in all of these articles is to first study \eqref{GFPE.Nemytskii} in terms of existence or uniqueness, respectively, and then transfer these results to \eqref{GMVSDE.Nemytskii} by means of the Ambrosio--Figalli--Trevisan superposition principle, which was first introduced in \cite{barbu2018prob,barbu2020fromNonlinearFPE}. The superposition principle relates solutions to \eqref{GFPE.Nemytskii} to weak solutions to \eqref{GMVSDE.Nemytskii}.
We refer to \cite{rehmeier2025mckean}, where the weak solutions constructed in this manner are shown to constitute a nonlinear Markov process in the sense of McKean.
Furthermore, in \cite{grube2021strong,grube2022strongTime,grube2026strong} it was shown that, under general conditions on the coefficients, the weak solutions constructed in \cite{barbu2021solutions,barbu2023timeDependent} are actually strong.
 
A different approach independent of the superposition principle was taken, e.g., in \cite{hao2021euler} and \cite{wang2023singular}.
As already mentioned in the previous subsection, in \cite{hao2021euler}, the authors considered \eqref{GMVSDE.Nemytskii} for $\sigma\equiv \sqrt{2}1_{d\times d}$ with a focus on a probabilistic approach, based on an Euler Scheme.
In \cite{wang2023singular}, \eqref{GMVSDE.Nemytskii} was studied using a fixed-point approach for a certain class of unbounded drift coefficients, which are uniformly Lipschitz continuous in their last variable, and general uniformly elliptic diffusion coefficients, both with density dependence. However, the analysis focused on diffusion coefficients with a nonlocal density dependence that does not cover Nemytskii-type diffusion coefficients.
We note that to study \eqref{GMVSDE} it is also possible to use the fixed-point technique in \cite{wang2023singular} instead of an Euler scheme approach, as we perform.

While writing this paper, the preprint \cite{djete2026nonexchangeablemeanfieldgames} appeared, where the author studied a mean field game problem with equilibrium dynamics described by an equation related to \eqref{GMVSDE.1}-\eqref{GMVSDE.2}. Reducing their more general equation (incorporating a general nondegenerate diffusion coefficient that is independent of the law of the solution, common noise, and quantities related to mean field games) to a comparable version of \eqref{GMVSDE}, their equation reads
\begin{align}\label{GMVSDE.uniform}
	dX(t)&= b\left(t,X(t),\left\langle \tilde{K}_1(U),\frac{d\sL_{(X(t),U)}}{d(x,u)}(X(t),\bullet)\right\rangle\right)dt+ \sqrt{2}dW(t),
\end{align}
where $U \overset{d}{\sim}\mathcal{U}([0,1])$, $W$ is a Wiener process independent of $(X(0),U)$, and ${d\sL_{(X(t),U)}}/{d(x,u)}$ denotes the 
Radon--Nikodym density of $\sL_{(X(t),U)}$ with respect to the $(d+1)$-dimensional Lebesgue measure.
Here, in comparison with \eqref{GMVSDE.1}-\eqref{GMVSDE.2}, the 'labels' of the process $X$ are encoded in the uniformly distributed random variable $U$.
Prior to \cite{djete2026nonexchangeablemeanfieldgames}, SDEs of this type, but without any distribution dependence in their coefficients, were introduced in \cite{lacker2023label-state} as mathematical substitutes for distribution independent SDEs similar to \eqref{GMVSDE.1}-\eqref{GMVSDE.2} driven by an infinite system of essentially pairwise independent (e.p.i.) Wiener processes $W^u, u \in [0,1]$. The formulation and study of \eqref{GMVSDE.1}-\eqref{GMVSDE.2} gives rise to technical difficulties, including the need to work on Fubini extensions, as we also do in the present work (cf. beginning of Section \ref{section.solutionConcepts.GMVSDE}).
Note that the corresponding Fokker--Planck equation to \eqref{GMVSDE.uniform} is the same as \eqref{GFPE} for the same choice of coefficients.
Among other results,  \cite{djete2026nonexchangeablemeanfieldgames} proves the existence of weak solutions to \eqref{GMVSDE.uniform}. In contrast to \cite{djete2026nonexchangeablemeanfieldgames}, we prove that their solutions are strong.
Moreover, we obtain a new existence result for weak solutions to \eqref{GMVSDE.uniform} and show that these solutions are, in fact, strong. Furthermore, we obtain new existence and uniqueness results for the corresponding Fokker--Planck equation that are not covered by \cite{djete2026nonexchangeablemeanfieldgames}, see Theorem~\ref{theorem.GMVSDE.weakExistence} and Theorem~\ref{theorem.GFPE.uniqueness}.

\subsection{Structure of the paper}
In Section \ref{section.notation}, we introduce the notation that will be used frequently throughout this paper.
In Section \ref{section.solutionConcepts}, we introduce the solution framework for \eqref{GMVSDE} and \eqref{GFPE}.
In Section \ref{section.mainresults}, we state our main results; more precisely, in Section \ref{section.mainresults.1} we give our main results regarding the existence and uniqueness of solutions to \eqref{GMVSDE}, in Section \ref{section.mainresults.2} we provide our existence and uniqueness results for \eqref{GFPE}, and in Section \ref{section.mainresults.3}, we investigate the influence of the initial condition and the interaction kernel on the solutions to \eqref{GFPE}.
In Section \ref{section.superpositionPrinciple}, we discuss the superposition principle relating solutions to \eqref{GFPE} to solutions to \eqref{GMVSDE}.
Section \ref{section.EulerScheme} is dedicated to the Euler scheme for \eqref{GMVSDE}.
In Sections \ref{section.proof.GMVSDE.weakExistence} and \ref{section.proof.claim.ArzelaAscoli.compactnessInL1}, we prove Theorem~\ref{theorem.GMVSDE.weakExistence}, our main result regarding the existence of a weak solution to \eqref{GMVSDE}.
In Section \ref{section.proof.GMVSDE.strongExistence}, we prove Theorem~\ref{theorem.GMVSDE.strongExistence}, in which we show that the constructed weak solution is, in fact, strong.
In Section \ref{section.proof.GMVSDE.strongExistence.uniform}, we prove Theorem~\ref{theorem.GMVSDE.strongExistence.uniform}, in which we construct strong solutions to \eqref{GMVSDE.uniform}.
In Section \ref{section.proof.GMVSDE.uniqueStrongSolution}, we prove Corollary~\ref{corollary.GMVSDE.uniqueStrongSolution}, which provides under slightly stronger assumptions, the existence of a unique strong solution to \eqref{GMVSDE}.
In Section \ref{section.proof.GFPE.uniqueness}, we prove Theorem~\ref{theorem.GFPE.uniqueness}, our main result on the uniqueness of Schwartz distributional solutions to \eqref{GFPE}.
In Section \ref{section.proof.GFPE.stepKernel}, we prove Theorem~\ref{theorem.GFPE.stepKernel}, which is a result on a specific influence of the initial condition and the interaction kernel on the solution to \eqref{GFPE}.
In Appendix \ref{appendix.heatKernel}, we recall properties of the classical heat kernel from \cite{hao2021euler}. In Appendix \ref{appendix.gronwall}, we recall a discrete Gronwall lemma from \cite{mckee1982gronwall}. In Appendix \ref{appendix.fubiniextension}, we recall the notion of Fubini extension and related aspects from \cite{sun2006exact}.

\subsection{Notation}\label{section.notation}
Let $(E,\mathcal{E},\mu)$ be a measure space. If $E$ is a topological space and $\mathcal{E}$ is the Borel $\sigma$-algebra on $E$, we write $\mathcal{E}=\sB(E)$. Let $\mu$ be a nonnegative measure on $(E,\mathcal{E})$. Let $Y$ be a separable Banach space, endowed with its Borel $\sigma$-algebra.
On $\R^d$ we denote the Lebesgue measure by $dx$, and the usual inner product by $x\cdot y$, where $x, y \in \R^d$.

\textit{(Generalized) Function spaces.} For $p\in [1,\infty]$, we denote the set of all (equivalence classes of $\mu$-a.e. coinciding) $p$-integrable functions by $L^p(E,\mathcal{E},\mu;Y)$, and consider them with the usual norm denoted by $\|\cdot\|_{L^p}\equiv\|\cdot\|_p$. If $E$ is a Borel measurable subset of $\R^n$ and $\mathcal{E}$ is the corresponding Borel $\sigma$-algebra (w.r.t. the subspace topology), and $\mu$ is the $d$-dimensional Lebesgue measure restricted to $\mathcal{E}$, we simply abbreviate $L^p(E,\mathcal{E},\mu;Y)=L^p(E;Y)$.

By $H^k=H^k(\R^d)$, $k\in \N$, we denote the standard Sobolev spaces consisting of all $2$-integrable $k$-times weakly differentiable functions $f:\R^d\to \R$. The corresponding topological dual spaces are denoted by $H^{-k}=H^{-k}(\R^d)$.
Specifically, for the case $k=1$, we denote by $_{H^{-1}}\langle \cdot,\cdot\rangle_{H^1}$ the usual pairing between $H^1$ and $H^{-1}$, where on $L^2(\R^d)\times L^2(\R^d)$, the pairing coincides with the usual $L^2$ inner product denoted by $\langle \cdot,\cdot\rangle_2$.
The inner product $\langle\cdot,\cdot\rangle_{-1}=\langle\cdot,\cdot\rangle_{H^{-1}}$ is defined as
\begin{align*}
	\langle f,g \rangle_{-1}=\langle (\mathrm{id}-\Delta)^{-1}f,g\rangle_2.
\end{align*}
The induced norm is denoted by $\|\cdot\|_{-1}=\|\cdot\|_{H^{-1}}$.
By $W^{1,2}([0,T];Y)$ we denote the set $\{f \in L^2([0,T];Y) : \frac{d}{dt} f \in L^2([0,T];Y)\}$, where $\frac{d}{dt}$ is taken in the sense of $X$-valued Schwartz distributions.
As usual, we denote the space of Schwartz distributions on $(0,T)\times\R^d$ by $\mathcal{D}'((0,T)\times\R^d)$.

If $E$ is a locally compact and a $\sigma$-compact Hausdorff space, $C(E;Y)$ denotes the set of all continuous functions $f:E\to Y$. We consider it as a Fr\'echet space with the usual metric induced by the local uniform convergence.
For $T\in (0,\infty)$, we write $C_0([0,T];\R^d)$ for the subset of functions $w\in C([0,T];\R^d)$ which satisfy $w(0)=0$.
On $C_{(0)}([0,T];\R^d)$, we introduce the following sigma algebras $\sB_t(C_{(0)}([0,T];\R^d)):=\sigma (\pi_s :s\leq t)$,
		where $\pi_s(w):=w(s)$, for all $w\in C_{(0)}([0,T];\R^d)$ and all $s\in [0,T]$.
For improved readability, we also use the following abbreviations throughout this paper:
\begin{align*}
	\mathcal{C}_{T}^d:=C([0,T];\R^d),\quad \mathcal{C}_{0,T}^d:=C_0([0,T];\R^d).
\end{align*}

\textit{Measure spaces.} Let $E$ be a metric space. Then $\mathcal{P}(E)$ denotes the set of all Borel probability measures on $E$. The narrow topology is the topology induced by the weak convergence of probability measures, i.e., the following: Let $\mu_n, \mu \in \mathcal{P}(E)$. We say that $\mu_n$ converges weakly to $\mu$, for $n\to\infty$, if for all bounded functions $\varphi\in C(E;\R)$ 
\begin{align}
	\int_E \varphi(x)\mu_n(dx) \to \int_E \varphi(x)\mu(dx),\text{ for } n\to \infty.
\end{align}
If not indicated otherwise, we consider $\mathcal{P}(E)$ as a measurable space together with its Borel $\sigma$-algebra.

\textit{Probability theory.} Let $(\Omega,\sF,\P)$ be a probability space and $(E,\mathcal{E})$ a measurable space. Let $X:\Omega\to E$ be a measurable map. We denote the distribution/law of $X$ on $E$ by $\sL_X:=\P\circ X^{-1}$, whenever this notation is unambiguous.
Furthermore, we say that $(\Omega,\sF,\P;(\sF_t))$ is a stochastic basis if $(\Omega,\sF,\P)$ is a complete probability space and $(\sF_t)$ is a right-continuous filtration on $(\Omega,\sF)$ augmented by the $\P$-zero sets.

For basic results on Fubini extensions and, in particular, the notion of \emph{essentially pairwise independent} (e.p.i.) random variables, we refer to Appendix~\ref{appendix.fubiniextension}.

\textit{Miscellaneous.} Let $a,b \in \mathbb{R}$. We write $a \lesssim_t b$ if there exists a constant $C_t>0$, depending on the parameter $t$, such that $a \leq C_t b$. Likewise, we write $a \lesssim b$ if there exists a constant $C>0$, independent of $t$, such that $a \leq Cb$.
\section{Definition of solutions to \texorpdfstring{\eqref{GMVSDE}}{GMVSDE} and \texorpdfstring{\eqref{GFPE}}{GFPE}}\label{section.solutionConcepts}
In this section, we introduce the solution framework for \eqref{GMVSDE} and \eqref{GFPE}.

\subsection{Solutions for \texorpdfstring{\eqref{GMVSDE}}{GMVSDE}}\label{section.solutionConcepts.GMVSDE}
As emphasized in previous work, including \cite{aurell2022stochastic} and \cite{coppini2025nonlinear}, it is desirable to construct a solution process $(u,\omega) \mapsto X^u(\omega)$ to \eqref{GMVSDE} that is jointly measurable on a suitable underlying probability space. Such joint measurability is usually not guaranteed on the standard product probability space. The reason for this is that \eqref{GMVSDE} calls for an uncountable family of essentially pairwise independent Wiener processes $(W^u)_{u\in[0,1]}$ such that $(u,\omega) \mapsto W^u(\omega)$ is measurable. As shown in \cite[Proposition 2.1]{sun2006exact}, such a process does not exist on a product space of type $([0,1]\times\Omega,\sB([0,1])\otimes\sF,du\otimes\P)$.
This issue is overcome by the notion of \textit{Fubini extension}. A Fubini extension is an extension of the canonical product space that preserves the Fubini property while allowing for the existence of essentially pairwise independent random variables with the required joint measurability, e.g., a family of Wiener processes with these properties. For details, we refer to \cite{sun2006exact} (see also \cite{aurell2022stochastic}).

For the definition of a Fubini extension of a product probability space, the notion of an essentially pairwise independent family of random variables, and corresponding technical results used throughout this work, we refer the reader to Appendix~\ref{appendix.fubiniextension}.

We introduce the following notion of a (probabilistically) weak solution to \eqref{GMVSDE}.
\begin{definition}[Weak solution to \eqref{GMVSDE}]\label{definition.GMVSDE.weakSolution}
A tuple 
\begin{align*}
    (([0,1]\times\Omega,\mathcal I\boxtimes\sF,\lambda\boxtimes\P),(\sF^u_\cdot)_{u\in [0,1]},(X^u)_{u\in[0,1]},(W^u)_{u\in [0,1]})
\end{align*}
(short: $(X,W)\equiv(X^u,W^u)_{u\in [0,1]}$) is called a (probabilistically) weak solution to \eqref{GMVSDE} with initial condition $(\nu_0^u)_{u\in [0,1]} \subseteq \sP(\R^d)$, if
\begin{enumerate}[label=(\roman*)]
    \item $([0,1]\times\Omega,\mathcal I \boxtimes\sF,\lambda\boxtimes\P)$ is a Fubini extension of $([0,1]\times\Omega,\mathcal I\otimes\sF,\lambda\otimes\P)$, where $(\Omega,\sF,\P)$ is a complete probability space and $(I,\mathcal{I},\lambda)$ a suitable complete extension of the classical (complete) Lebesgue space $([0,1],\mathcal A,du)$;
    \item $(\sF^u_\cdot)_{u\in [0,1]}$ is a family of right-continuous filtrations on $(\Omega,\sF)$ augmented by the $\P$-zero sets;
    \item $[0,1]\times\Omega\ni(u,\omega)\mapsto (X^u(\omega), W^u(\omega)) \in \mathcal{C}_{T}^d\times \mathcal{C}_{0,T}^d$ is $\mathcal{I}\boxtimes\sF \slash \sB(\mathcal{C}_{T}^d)\otimes \sB(\mathcal{C}_{0,T}^d)$-measurable;
    \end{enumerate}
    Regarding (iv)-(vii), for $\lambda$-a.e. $u \in [0,1]$,
    \begin{enumerate}[label=(\roman*)]\setcounter{enumi}{3}
    \item $X^u, W^u$ are $(\sF^u_t)$-adapted stochastic processes;
    \item  $(W^u)_{u \in [0,1]}$ is essentially pairwise independent (e.p.i.) w.r.t. $\lambda$;
    \item  $\sL_{(W^u,X^u(0))}=P^W \otimes \sL_{X^u(0)}$, where $P^W$ denotes the Wiener measure on \\
    $(\mathcal{C}_{0,T}^d,\sB(\mathcal{C}_{0,T}^d))$;
    \item $W^u$ is a standard $d$-dimensional $(\sF^u_t)$-Wiener process;
    \item For all $t\in[0,T]$, $[0,1]\ni u\mapsto \sL_{X^u}\in\sP(C_T^d)$ is $\mathcal A\slash \sB(\sP(C_T^d))$-measurable;
    \item $[0,1] \ni u\mapsto \nu_0^u \in \sP(\R^d)$ is Borel measurable, $\sL_{X^u(0)}(dx)du=\nu_0^u(dx)du$ as Borel measures on $\R^d\times [0,1]$;
    \item For all $t \in (0,T]$, $\sL_{X^u(t)}(dx)du\! =\! p^u_t(x)dxdu$ as Borel measures on $\R^d\times [0,1]$, for some nonnegative function $[[0,T]\times \R^d\times [0,1] \ni (t,x,u)\mapsto p_t^u(x)] \in L^1_{\mathrm{loc}}([0,T]\times\R^d\times [0,1])$;
    \item For $i=1,2$: $\langle \tilde{K}_i(u),p_t^\bullet(X^u(\omega)(t))\rangle <\infty$, for $dt\otimes(\lambda\boxtimes\P)$-a.e. $(t,u,\omega)\in [0,T]\times[0,1]\times\Omega$;
    \item \begin{align*}
    	& \E_\boxtimes\int_0^T\left|b\Big(s,X^\cdot(\cdot)(s),  \langle\tilde{K}_1(\cdot), p^\bullet_s(X^\cdot(\cdot)(s))\rangle  \Big)\right|ds \\
    	&+ \E_\boxtimes\int_0^T \left|\sigma\Big(s,X^\cdot(\cdot)(s),  \langle\tilde{K}_2(\cdot), p^\bullet_s(X^\cdot(\cdot)(s))\rangle\Big)\right|^2ds <\infty,
    	\end{align*}
    	where $E_\boxtimes$ denotes the expectation with respect to the measure $\lambda\boxtimes \P$;
    \begin{align*}
        &X^u(\omega)(t) 
        = X^u(\omega)(0)
        + \int_0^t b\Big(s,X^u(\omega)(s),  \langle \tilde{K}_1(u), p^\bullet_s(X^u(\omega)(s))\rangle  \Big) ds\\
        &\quad + \int_0^t \sigma\Big(s,X^u(\cdot)(s),  \langle \tilde{K}_2(u), p^\bullet_s(X^u(\cdot)(s))\rangle  \Big) d W^u(\cdot)(s)\ (\omega)\qquad  \forall t\in [0,T].\notag
    \end{align*}
\end{enumerate}
\end{definition}
We have the following definition of a strong solution to \eqref{GMVSDE}.

\begin{definition}[Strong solution to \eqref{GMVSDE}]\label{definition.GMVSDE.strongSolution}
Let $(\nu_0^u)_{u\in [0,1]}\subseteq \mathcal{P}(\R^d)$.
We say that \eqref{GMVSDE} has a (probabilistically) strong solution with initial condition $(\nu_0^u)_{u\in [0,1]}$ if there exists a map (strong solution functional)
$$F:\R^d\times [0,1]\times \mathcal{C}_{0,T}^d\to \mathcal{C}_{T}^d,$$
 which is
$\overline{\sB(\R^d) \otimes\sB([0,1])\otimes\sB(\mathcal{C}_{0,T}^d)}^{\nu_0^udu \otimes P^W}\slash \sB(\mathcal{C}_{T}^d)$-measurable
\footnote{Here,
$\overline{\sB(\R^d)\otimes\sB([0,1])\otimes\sB(\mathcal{C}_{0,T}^d)}^{\nu_0^udu\otimes P^W}$
denotes the completion of
$\sB(\R^d)\otimes \sB([0,1])\otimes\sB(\mathcal{C}_{0,T}^d)$
with respect to the measure $\nu_0^udu\otimes P^W$.},
such that for
$\nu_0^u du$-a.e. $(x,u)\in\R^d\times[0,1]$, the map $F(x,u,\cdot)$ is
$\overline{\sB_t(\mathcal{C}_{0,T}^d)}^{P^W}/\sB_t(\mathcal{C}_{T}^d)$-measurable for every
$t\in[0,T]$, and 
for every Fubini extension $([0,1]\times\Omega,\mathcal{I}\boxtimes\sF,\lambda\boxtimes\P)$ with family of filtrations $(\sF_\cdot^u)_{u\in [0,1]}$ on $(\Omega,\sF)$, as in Definition \ref{definition.GMVSDE.weakSolution} (i), (ii), that supports a measurable map $(u,\omega)\mapsto (X_0^u(\omega), W^u(\omega))$ which satisfies the conditions in Definition \ref{definition.GMVSDE.weakSolution} (iii)-(vii), (ix),
\begin{align*}
    (([0,1]\times\Omega,\mathcal I \boxtimes\sF,\lambda\boxtimes\P),(\sF^u_\cdot)_{u\in [0,1]},(X^u)_{u\in[0,1]},(W^u)_{u\in [0,1]}),
\end{align*}
where $X^u:=F(X_0^u,u,W^u)$, is a weak solution to \eqref{GMVSDE} with $X(0)=X_0$ $\lambda\boxtimes\P$-a.s.
\end{definition}

\begin{definition}[Pathwise uniqueness]\label{definition.GMVSDE.pathwiseUniqueness}
    We say that pathwise uniqueness holds for \eqref{GMVSDE} if for every two weak solutions $(X,W), (\bar{X},W)$ defined on the same probability space $([0,1]\times\Omega,\mathcal I \boxtimes\sF,\lambda\boxtimes\P)$, $X(0)=\bar{X}(0)$ $\lambda\boxtimes\P$-a.e., implies $X(t)=\bar{X}(t)$ for all $t\in [0,T]$, $\lambda\boxtimes \P$-a.e.
\end{definition}
\subsection{Solutions to \texorpdfstring{\eqref{GFPE}}{GFPE}}\label{section.solutionConcepts.GFPE}

Let $a:[0,T]\times \R^d\times [0,\infty)\to S_+(\R^d)$ be a Borel measurable map, where $S_+(\R^d)$ denotes the set of all symmetric, nonnegative definite matrices in $\R^{d\times d}$.

In the following, $\mathcal{P}(\R^d)$ is always considered together with the narrow topology and we use Einstein summation convention.
\begin{definition}[Solution to \eqref{GFPE}]\label{definition.GFPE.solution}
    A family of families $(p_t^u)_{t\in [0,T]}\subseteq {(L^1\cap\mathcal{P})(\R^d)}$, $u \in [0,1]$, is said to be a (Schwartz distributional) solution to \eqref{GFPE} with initial condition $(\nu_0^u)_{u\in [0,1]}\subseteq \mathcal{P}(\R^d)$, or $\left.p^u\right|_{t=0}=\nu_0^u\in \mathcal{P}(\R^d)$, $u\in [0,1]$, if
    \begin{enumerate}[label=(\roman*)]
        \item $[0,1]\ni u \mapsto \nu_0^u \in \mathcal P (\R^d)$ is Borel measurable;
        \item $[[0,T]\times\R^d\times [0,1]\ni(t,x,u)\mapsto p^u_t(x)]\in L^1_{\mathrm{loc}}([0,T]\times\R^d\times [0,1])$;
        \item $[0,T]\ni t\mapsto 1_{\{0\}}(t) \nu_0^udu + 1_{(0,T]}(t)p^u_t(x)dxdu \in \mathcal{P}(\R^d\times[0,1])$ is continuous;
        \item For i=1,2: $\langle\tilde{K}_i(u), p^\bullet_t(x)\rangle <\infty$, for $p_t^u(x)dxdudt$-a.e. $(x,u,t) \in \R^d\times (0,1)\times(0,T)$;
        \item \begin{align*}\int_0^T\int_0^1 \int_{\R^d}& \Big[\left|b^i\left(t,x,\langle\tilde{K}_1(u), p^\bullet_t(x)\rangle\right)\right|\\
        &\quad + \left|a^{ij}\left(t,x, \langle\tilde{K}_2(u), p^\bullet_t(x)\rangle\right)\right| \Big]p^u_t(x)dx du dt<\infty;\end{align*}
        \item For all $\varphi \in C_c^\infty([0,T)\times \R^d\times (0,1))$,
    \begin{align*}
        &\int_0^T\int_0^1\int_{\R^d} \Big[\partial_t\varphi(t,x,u)
        + b\left(t,x,\langle\tilde{K}_1(u), p^\bullet_t(x)\rangle\right)\cdot \nabla_x \varphi(t,x,u)\\
        &\quad \quad \quad \quad \quad \quad + a^{ij}\left(t,x,\langle\tilde{K}_2(u), p^\bullet_t(x)\rangle\right)\partial_{x_i}\partial_{x_j}\varphi(t,x,u)\Big] p^u_t(x)dxdudt\\
        &+\int_0^1\int_{\R^d}\varphi(0,x,u)\nu_0^u(dx)du = 0.
    \end{align*}
    \end{enumerate}
    
    Furthermore, we say that two solution $(p_t^u)_{t\in [0,T]},(\bar{p}_t^u)_{t\in [0,T]}, u \in [0,1],$ to \eqref{GFPE} with the same initial condition $(\nu_0^u)_{u\in [0,1]}$ are unique if
    \begin{align*}
    	p_t^u(x)dxdu=\bar{p}_t^u(x)dxdu \quad \forall t \in [0,T].
    \end{align*}
\end{definition}

\section{Main results}\label{section.mainresults}
We introduce the following set of assumptions.
Let $[0,1]\ni u\mapsto \nu^u_0 \in \sP(\R^d)$ be a Borel measurable map.

\begin{enumerate}[label=(\roman*)]
    \item[(H1)] For all $u \in [0,1]$, $\nu_0^u\ll dx$ with $\nu_0^u(dx)=v_0^u(x)dx$ for some Borel measurable function $v_0\equiv v_0^\cdot(\cdot):[0,1]\times\R^d\to [0,\infty)$. Furthermore, there exists $C\in (0,\infty)$ such that
  $$\|v_0\|_{L^\infty([0,1]\times\R^d)} \leq C.$$
    \item[(H2)]  $L^1([0,1])\ni\varphi \mapsto (u\mapsto \langle\tilde{K}_1(u),\varphi\rangle) \in L^1([0,1])$ is a bounded (linear) operator, and there exists $\alpha_0 \in (0,1]$ such that
    	\begin{align}\label{H2:1}
		\int_0^1\|\tilde{K}_1(u+\kappa)-\tilde{K}_1(u)\|_{\mathrm{TV}}^{\alpha_0} du\to 0,
		\text{as }|\kappa|\to 0,
	\end{align}
	where we set $\tilde{K}_1(u+\kappa)\equiv \tilde{K}_1(u)$, whenever $u+\kappa\notin [0,1]$.
    \item[(H3)] $L^2([0,1])\ni\varphi \mapsto (u\mapsto \langle\tilde{K}_1(u),\varphi\rangle) \in L^2([0,1])$ is a bounded (linear) operator.
\end{enumerate}
\begin{remark}\label{remark.hypotheses}
\begin{enumerate}[label=(\roman*)]
	\item Assume that there exists $\tilde{k}_1\in L^\infty([0,1]\times[0,1])$ such that, for all $u \in [0,1]$,
\begin{align*}
	\tilde{K}_1(u)=\tilde{k}_1(u,v)dv.
\end{align*}
	Then (H2) holds. In particular, \eqref{H2:1} holds, since
	\begin{align*}
		\int_0^1\|\tilde{K}_1(u+\kappa)-\tilde{K}_1(u)\|_{\mathrm{TV}}du
		=\int_0^1\int_0^1|\tilde{k}_1(u+\kappa,v)-\tilde{k}_1(u,v)|dvdu\to 0,
	\end{align*}
    as $|\kappa|\to 0$.
	\item Assume that $(\nu_0^u)_{u\in [0,1]}$ satisfies Condition (H1). Then
	\begin{align*}
		\int_0^1\int_{\R^d}|v_0^{u+\kappa}(x)-v_0^u(x)| dxdu\to 0\text{, as }|\kappa|\to 0,
	\end{align*}
	where we set $v_0^{u+\kappa}\equiv v_0^u$, whenever $u+\kappa\notin [0,1]$.
\end{enumerate}
\end{remark}
\subsection{Weak and strong solutions to \texorpdfstring{\eqref{GMVSDE}}{GMVSDE}}\label{section.mainresults.1}
The following result provides sufficient conditions to conclude the existence of a weak solution to \eqref{GMVSDE} and, consequently, a solution to \eqref{GFPE}.
\begin{theorem}\label{theorem.GMVSDE.weakExistence}
    Assume that Conditions (H1), (H2) hold. Let $b\in L^\infty([0,T]\times\R^d;\R^d), \sigma\equiv \sqrt{2}\cdot\mathrm{1}_{d\times d}$ such that for some constant $C>0$ and $\alpha\in (0,1]$
    \begin{align}\label{theorem.weakExistence:condition.b}
        |b(t,x,r)-b(t,x,\bar{r})|\leq C|r-\bar{r}|^\alpha,\quad \forall (t,x,r,\bar{r}) \in [0,T]\times\R^d\times \R\times\R.
    \end{align}
    Then \eqref{GMVSDE} has a probabilistically weak solution with initial condition $(\nu_0^u)_{u \in [0,1]}$.
    Furthermore, $(p_t^u)_{t\in[0,T]}(=\sL_{X^u(t)}(dx)/dx), u \in [0,1]$, is a solution to \eqref{GFPE} for $a\equiv 1_{d\times d}$ and with initial condition $(\nu_0^u)_{u \in [0,1]}$.
\end{theorem}
By essentially proving a (restricted) Yamada--Watanabe theorem on Fubini extensions, the classical pathwise uniqueness result for SDEs with bounded drift coefficients and constant diffusion coefficients \cite{veretennikov1980strong} implies the following theorem.
\begin{theorem}\label{theorem.GMVSDE.strongExistence}
    Under the assumptions of Theorem \ref{theorem.GMVSDE.weakExistence}, there exists a probabilistically strong solution to \eqref{GMVSDE} with initial condition $(\nu_0^u)_{u\in [0,1]}$, with corresponding strong solution functional $F$. In particular, the weak solution $(X,W)$ provided by Theorem \ref{theorem.GMVSDE.weakExistence} can be written as $X^\cdot=F(X^\cdot(0),\cdot,W^\cdot)$ a.s.
\end{theorem}
The strong solution functional obtained in the previous theorem (cf. Definition \ref{definition.GMVSDE.strongSolution}) can be used to construct probabilistically strong solutions to \eqref{GMVSDE.uniform}, as the following theorem shows.
\begin{theorem}\label{theorem.GMVSDE.strongExistence.uniform}
    Consider the situation of Theorem~\ref{theorem.GMVSDE.strongExistence} and let $F$ denote the constructed strong solution functional for \eqref{GMVSDE} with initial condition $(\nu_0^u)_{u\in [0,1]}$.
    We set $\bar{\nu}_0(dx,du)=\nu_0^u(dx)du (\in \mathcal{P}([0,1]\times \R^d))$. 
    Then, for every standard $d$-dimensional $(\sF_t)$-Wiener process $W$, every $\sF_0$-measurable $X_0:\Omega\to \R^d$, and every $\sF_0$-measurable random variable $U:\Omega \to [0,1]$ that is uniformly distributed on $[0,1]$ with 
    $\sL_{(X_0,U)}(dx,du)=\nu_0(dx, du)$, with $W$ independent of $(X_0,U)$, all defined on a common stochastic basis $(\Omega,\sF,\P;(\sF_t)_{t\in [0,T]})$, the process $((F(X_0,U,W),U),W)$ is a (classical) probabilistically weak solution to \eqref{GMVSDE.uniform} with initial condition $(X_0,U)$.
\end{theorem}

Imposing stronger conditions on $b$ and the kernel $\tilde{K}_1$, we obtain the uniqueness of the strong solution to \eqref{GMVSDE} provided by Theorem \ref{theorem.GMVSDE.strongExistence}.

\begin{corollary}\label{corollary.GMVSDE.uniqueStrongSolution}
Let $d\geq 3$.
Assume that Conditions (H1), (H2) and (H3) hold. Let $b\in L^\infty([0,T]\times\R^d\times \R;\R^d)$ such that for some $C>0$
\begin{align}\label{corollary.GMVSDE.uniqueStrongSolution:condition.b}
    |b(t,x,r)-b(t,x,\bar{r})| \leq C|r-\bar{r}|\quad \forall (t,x,r,\bar{r}) \in [0,T]\times\R^d\times \R\times\R,
\end{align}
and $\sigma\equiv \sqrt{2}\cdot\mathrm{1}_{d\times d}$.
Then there exists a unique probabilistically strong solution to \eqref{GMVSDE} with initial condition $(\nu_0^u)_{u\in [0,1]}$, that is,
the strong solution to \eqref{GMVSDE} with initial condition $(\nu_0^u)_{u\in [0,1]}$ provided by Theorem \ref{theorem.GMVSDE.strongExistence} is pathwise unique among all weak solution to \eqref{GMVSDE} with initial condition $(\nu_0^u)_{u\in [0,1]}$.
\end{corollary}

\subsection{Uniqueness of solutions to \texorpdfstring{\eqref{GFPE}}{GFPE}}\label{section.mainresults.2}
For the uniqueness of Schwartz distributional solutions to \eqref{GFPE}, we may consider a more general diffusion coefficient, as in the previous subsection.

We consider the following special case of \eqref{GFPE}: For $u\in [0,1]$,
\begin{align}
    \partial_t p^u_t(x)  = -\mathrm{div}\left(b\Big(t,x,  \langle\tilde{K}_1(u), p^\bullet_s(x)\rangle  \Big)p_t^u(x)\right)
    + \Delta \beta (p^u_t(x)),\quad (t,x) \in [0,T]\times\R^d, \label{GFPE.uniqueness}
\end{align}
where $\beta:\R\to\R$ is a given function that satisfies the following condition:
\begin{enumerate}[label=(\roman*)]
	\item [(H4)] $\beta \in C^1(\R)$, $\beta(0)=0$, and there exists $\gamma_0>0$ such that
		\begin{align*}
			\gamma_0 |r_1-r_2|^2 \leq (\beta(r_1)-\beta(r_2))\cdot(r_1-r_2) \quad \forall r_1,r_2 \in \R.
		\end{align*}
\end{enumerate}
Note that, in the above formulation, we only consider the case of an isotropic, time- and space-homogeneous noise, which corresponds to replacing the kernel $\tilde{K}_2(u,dv)$ in \eqref{GFPE} by the Dirac measure $\delta_u(dv)$.
\begin{theorem}\label{theorem.GFPE.uniqueness}
    Let $d\geq 3$.
    Assume that Conditions (H3), (H4) hold. Let $b\in L^\infty([0,T]\times\R^d\times \R;\R^d)$ such that for some $C>0$
    \begin{align}\label{theorem.GFPE.uniqueness:1}
        |b(t,x,r)-b(t,x,\bar{r})| \leq C|r-\bar{r}|\quad \forall (t,x,r,\bar{r}) \in [0,T]\times\R^d\times \R\times\R.
    \end{align}
    Then there is at most one solution to \eqref{GFPE.uniqueness} with initial condition $(\nu_0^u)_{u\in [0,1]}$, in the class $(\mathcal P\cap L^\infty)([0,T]\times\R^d\times [0,1])$.
    
    Moreover, if $\beta=\mathrm{id}_{\R}$ and, additionally,  Conditions (H1) and (H2) hold, then there is exactly one solution in the above mentioned class.
\end{theorem}
\subsection{Impact of the initial condition \texorpdfstring{$\nu_0$}{nu0} and the interaction kernel \texorpdfstring{$\tilde{K}_1$}{K1} on the solution to \texorpdfstring{\eqref{GFPE}}{GFPE}}\label{section.mainresults.3}
Consider the situation of Theorem~\ref{theorem.GMVSDE.weakExistence}, and assume, additionally, that the initial condition $(\nu_0^u)_{u\in[0,1]}$ and interaction kernel $(\tilde{K}_1(u))_{u\in[0,1]}$, viewed as maps
\begin{align*}
	u\mapsto \nu_0^u,\quad u \mapsto \tilde{K}_1(u),
\end{align*}
are piecewise constant. Then the following theorem shows that the one-dimensional time marginal law density of the constructed probabilistically weak solution to \eqref{GMVSDE} at time $t\in (0,T]$ and evaluated in $y\in\R^d$, that is, viewed as the function
\begin{align*}
	u\mapsto p^u_t(y),
\end{align*}
is also piecewise constant.
In \cite[Theorem 2.1]{coppini2022note}, a related result was obtained for a class of Fokker--Planck equations on $\R$ associated to equation \eqref{GMVSDE.weakly.1}-\eqref{GMVSDE.weakly.2} with additive noise, and drift coefficient which depends non-locally on its distribution variable. However, their method of proof cannot be adapted to our framework  because of the structural differences between the coefficients of their Fokker--Planck equations and those of \eqref{GFPE}, see Section \ref{section.introduction} for details.

\begin{theorem}\label{theorem.GFPE.stepKernel}
	Assume that all conditions of Theorem~\ref{theorem.GMVSDE.weakExistence} hold. Furthermore, assume that there exist an at most countable index set $J$ and a family of pairwise disjoint $A_j \in \sB([0,1])$, $j\in J$, such that $\bigcup_{j\in J} A_j = [0,1]$, and that
	\begin{align*}
		A_j\ni u\mapsto(\nu_0^u,\tilde{K}_1(u))
	\end{align*}
	is constant, for all $j\in J$.
	Let $(p_t^u)_{t\in [0,T]}, u \in [0,1],$ denote the solution to \eqref{GFPE} for $a\equiv 1_{d\times d}$ and with initial condition $(\nu_0^u)_{u\in [0,1]}$ provided by Theorem~\ref{theorem.GMVSDE.weakExistence}.
	Then, for every $t \in (0,T]$ and $y\in \R^d$,
	\begin{align*}
		\int_{A_j}\int_{A_j} |p^u_t(y)-p^v_t(y)|dvdu = 0\qquad \forall j\in J.
	\end{align*}
	In particular, the (continuous) map $(0,T]\times\R^d\ni (t,y)\mapsto \left.p_t^\cdot(y)\right|_{A_j} \in L^1(A_j)$ takes values in the classes of $du$-a.e. constant functions on $A_j$, for all $j \in J$.
\end{theorem}
\begin{remark}
	The assumptions in Theorem~\ref{theorem.GFPE.stepKernel} are, for example, satisfied for kernels of the form
		\begin{align*}
			\tilde{K}_1(u,A)=\sum_{i\in J}c_i1_{A_i}(u) \int_A f(v)dv,\quad A\in \sB([0,1]),
			\end{align*}
		where $f\in L^\infty([0,1])$, $f\geq 0$, $\sum_{i\in J}c_i <\infty$, $c_i\geq 0$, for all $i\in J$. Here, $J$ is any index set and $(A_j)_{j\in J}$ is any Borel partition of $[0,1]$ as required by Theorem~\ref{theorem.GFPE.stepKernel}.
\end{remark}
\section{On the superposition principle}\label{section.superpositionPrinciple}
Following the strategy in \cite{barbu2018prob,barbu2020fromNonlinearFPE}, one could first aim to solve \eqref{GFPE} in order to solve \eqref{GMVSDE}.
In the present framework, this strategy is not straightforward and involves several technical difficulties.
We briefly sketch how such a strategy could be pursued. 

Assume for the moment that there exists a Schwartz distributional solution $(p^u_t)_{t\in [0,T]}$, $u \in [0,1]$, to \eqref{GFPE} with $a_{ij}=\frac{1}{2}(\sigma\sigma^*)_{ij}$. Then, fixing $p^\cdot_t$ in the coefficients of \eqref{GFPE}, we obtain a linear Fokker--Planck equation for $du$-a.e. $u\in [0,1]$, i.e.,
\begin{align}
     \partial_t p^u_t(x) & = -\mathrm{div}\big(\tilde{b}^u(t,x)p_t^u(x)\big)\notag + \partial_{x_i}\partial_{x_j}\big(\tilde{a}_{ij}^u(t,x)p_t^u(x)\big),
\end{align}
where, for $i,j\in\{1,\dots,d\}$,
\begin{align*}
    \tilde{b}^u(t,x):=b\left(t,x, \langle\tilde{K}_1(u), p^\bullet_t(x)\rangle \right),\quad \tilde{a}^u_{ij}(t,x):=a_{ij}\left(t,x, \langle\tilde{K}_2(u), p^\bullet_t(x)\rangle \right).
\end{align*}
Note that for each such $u \in [0,1]$ there is a version of $t\mapsto p^u_t\in L^1(\R^d)$ such that the map $t\mapsto p^u_t(x)dx$ is narrowly continuous (i.e. continuous with respect to the narrow topology), see \cite[Lemma 2.3]{rehmeier2022flow}. In the following, we consider only this version and we will not change notation.
To each of these equations, we may apply the \textit{Ambrosio--Figalli--Trevisan-superposition principle} \cite{trevisan2016well-posedness} and obtain a probabilistically weak solution $(X^u, W^u)$ to
\begin{align}\label{superpositionPrinciple.SDE.fixed.u}
     d X^u_t   = \tilde{b}^u(t,X^u_t) dt+ \tilde{\sigma}^u(t,X^u_t) d W^u_t,
\end{align}
with $\sL_{X_t}(dx)=p^u_t(x)dx$ for all $t\in [0,T]$. This means that $(X^u,W^u)$ is a classical probabilistically weak solution to the equation  in \eqref{GMVSDE} which is indexed by this very $u$.
Moreover, if one can prove that the map 
\begin{align}\label{superpositionPrinciple.measurability}
	[0,1]\ni u\mapsto \sL_{(X^u,W^u)} \in \mathcal{P}(\mathcal{C}_T^d\times \mathcal{C}_{0,T}^d)
\end{align}
is Lebesgue measurable, then, by Theorem~\ref{appendix.fubiniextension.theorem.existence}, there exists a probability space $([0,1],\mathcal{I}, \mathcal{\lambda})$ extending the usual Lebesgue space, a probability space $(\Omega,\sF,\P)$, and a Fubini extension $([0,1]\times\Omega,\mathcal{I}\boxtimes\sF,\lambda\boxtimes \P)$ of $([0,1]\times\Omega,\mathcal I\otimes\sF,\lambda\otimes\P)$, and a family of processes $(\tilde{X}^u,\tilde{W}^u)$, $u\in [0,1]$, such that, for $\lambda$-a.e. $u\in[0,1]$,
\begin{align}\label{superpositionPrinciple.equality}
\sL_{(\tilde{X}^u,\tilde{W}^u)}=\sL_{(X^u,W^u)}.
\end{align}
In this case, the tuple $(\tilde{X},\tilde{W})$ can be proved to be a probabilistically weak solution to \eqref{GMVSDE} in the sense of Definition \ref{definition.GMVSDE.weakSolution}.
However, verifying the measurability of \eqref{superpositionPrinciple.measurability}, in general, might be a challenging task and it cannot be read off \eqref{GFPE} directly.

At this point, we would like to mention that in \cite{djete2026nonexchangeablemeanfieldgames} the superposition principle was applied to the subdomain $(0,1)\times\R^d$ of $\R^{d+1}$, which is justified by \cite{krasovitskii2022spp}. However, treating the boundary of this domain crucially relies on the existence of a suitable Lyapunov function, which has to be verified and which is not guaranteed in general.

\section{The Euler scheme}\label{section.EulerScheme}
In this section, we tailor the general Euler Scheme and related results from \cite{hao2021euler} to our framework.
Within this section and especially in the proof of Theorem~\ref{theorem.GMVSDE.weakExistence}, we use many properties and estimates of the classical heat kernel
	\begin{align}
       g(t,x):= (4\pi t)^{-d/2}e^{-|x|^2/(4t)},\quad t>0,x\in\R^d,
\end{align}
which can be found in \cite{hao2021euler}. For the convenience of the reader, we recall them in Appendix \ref{appendix.heatKernel}.

Let $T>0$. Throughout this subsection, we assume that $b\in L^\infty([0,T]\times\R^d\times\R;\R^d)$, $[0,1]\ni u\mapsto \nu^u_0 \in \sP(\R^d)$ is a Borel measurable map, and $\tilde{K}_1$ is considered as in Section \ref{section.goals} (e.g., satisfying assumption (H2) or (H3)).
Let $(X_0^u)_{u\in [0,1]}$ be a family of $\sF_0$-measurable random variables with values in $\R^d$ and distribution $\sL_{X_0^u}(dx)=\nu_0^u(dx)$ such that $(u,\omega)\mapsto X_0^u(\omega)$ is $\sB([0,1])\otimes \sF$-measurable, and $W$ is an $(\sF_t)$-Wiener process on $\R^d$. All these random variables are  assumed to be supported on a common stochastic basis $(\Omega,\sF,\P;(\sF_t)_{t\in [0,T]})$.

Let $N\in \mathbb{N}$ and set $h := T/N$. For $t\in [0,h]$ and $u \in [0,1]$, we set
\begin{align}\label{EulerScheme.1}
    X^{u,N}(t):= X_0^u + \sqrt{2} W(t).
\end{align}
For $t\in [kh,(k+1)h]$, with $k=1,\dots, N$, we inductively define $X^{u,N}(t)$ by
\begin{align}\label{EulerScheme.2}
    X^{u,N}(t):=\ & X^{u,N}(kh)
    + \int_{kh}^t b(s,X^{u,N}(kh),\langle \tilde{K}_1(u), p^{\bullet,N}_{kh}(X^{u,N}(kh))\rangle)ds \notag\\
    \quad &+\sqrt{2}(W(t)-W(kh)),
\end{align}
where $p^{u,N}_{kh}$ denotes the Radon--Nikodym density of $\sL_{X^{u,N}(kh)}(dx)$ with respect to Lebesgue measure.

The rest of this section is devoted to verify that this Euler scheme is indeed well-defined, to recall from \cite{hao2021euler} several useful estimates, and to prove certain regularity of the functions $(t,u,y)\mapsto p_{t}^{u,N}(y)$.
We set
\begin{align*}
	    \phi_N(t) &:= \sum_{j=0}^\infty jh 1_{(jh, (j+1)h]}(t),\quad t\in [0,\infty).
\end{align*}
Then \eqref{EulerScheme.1}-\eqref{EulerScheme.2} can be rewritten in the following form:
\begin{align}\label{GMVSDE.EulerScheme.1}
    X^{u,N}(t)= X_0^u + \int_0^t b^{u,N}(s,X^{u,N}(\phi_N(s)))ds + \sqrt{2}W(t),
\end{align}
where
\begin{align}
    b^{u,N}(s,x)&:= {1}_{s\geq h} b(s,x,\langle \tilde{K}_1(u),p_{\phi_N(s)}^{\bullet, N} (x)\rangle).
\end{align}
Furthermore, for $x\in\R^d$, we consider the process $X^{u,N,x}$ defined via the Euler scheme:
\begin{align}\label{EulerScheme}
	X^{u,N,x}(t)= x+ \int_0^t b^{u,N}(s,X^{u,N,x}({\Phi_N(s)}))ds +\sqrt{2}dW(t).
\end{align}
\begin{lemma}[{cf. \cite[Lemma 2.2]{hao2021euler}}]\label{lemma.DuhamelsFormula}
	For each $t \in (0,T]$, $x\in \R^d$, and $u \in [0,1
]$, $\sL_{X^{u,N,x}(t)}(dy)\ll dy$ and its density $p^{u,N,x}_t$ satisfies the following Duhamel formula
	\begin{align*}
		p^{u,N,x}_t(y)=g(t, y-x)+ \int_0^t \E[b^{u,N}(s,X^{u,N,x}(\phi_N(s)))\cdot \nabla g(t-s,y-X^{u,N,x}(s))]ds.
	\end{align*}
\end{lemma}

\begin{remark}[{cf. \cite[Remark 2.3]{hao2021euler}}]\label{remark.DuhamelsFormula}
	Let $u \in [0,1]$. For a general $\sF_0$-measurable random variable $X_0^{u,N}\equiv X_0^u:\Omega \to \R^d$ and each $t\in (0,T]$, note that, for each $x\in\R^d$, $X^{u,N,x}(t)$ is independent of $X_0^u$. Consequently, the law of $X^{u,N}(t)$ in \eqref{EulerScheme} is also absolutely continuous with respect to Lebesgue measure, with density given by
	\begin{align*}
		p^{u,N}_t(y)=\int_{\R^d}p^{u,N,x}_t(y)\ \P\circ (X_0^u)^{-1}(dx)\quad \forall (t,y) \in [0,T]\times\R^d.
	\end{align*}
\end{remark}
Since $\|b^{u,N}\|_{L^\infty}\leq \|b\|_{L^\infty}$, for all $u\in [0,1], N\in \mathbb{N}$, one obtains the subsequent lemma.
\begin{lemma}[{cf. \cite[Theorem 2.4]{hao2021euler}}]\label{lemma.heatkernelestimate.1}
    For all $T>0$, there exists $C=C(d,T,\|b\|_{L^\infty})>0$ such that for all $u\in [0,1], N\in \mathbb{N},\ x,y\in \R^d,\ t\in (0,T]$ :
 \begin{align}
     p^{u,N,x}_t(y) \leq C g(4t, x-y).
 \end{align}
\end{lemma}

The following lemma is a consequence of Lemma~\ref{lemma.heatkernelestimate.1} and Lemma~\ref{appendix.lemma.hao2021euler.propertiesOfFundamentalSolution.2}, which is obtained in the same way as \cite[Corollary 2.5]{hao2021euler} is concluded from \cite[Theorem 2.4]{hao2021euler} and Lemma~\ref{appendix.lemma.hao2021euler.propertiesOfFundamentalSolution.2}.
 \begin{lemma}[{cf. \cite[Corollary 2.5]{hao2021euler}}]\label{lemma.heatkernelestimate.2}
 Consider the situation of Remark~\ref{remark.DuhamelsFormula}. For each $u \in [0,1]$, we set $\nu_0^u(dx):= \sL_{X_0^u}(dx)$. Then,
     for all $T>0, \beta\in (0,1)$, there exists $C=C(d,T,\|b\|_{L^\infty}, \beta)>0$ such that for all $N\in \mathbb{N},\ u \in [0,1],\ x,y\in \R^d,\ t\in (0,T]$
 \begin{enumerate}[label=(\roman*)]
     \item $p_t^{u,N}(y) \leq C \int g(4t,x-y)\nu_0^u(dx)$;
     \item $|p_t^{u,N}(y_1)-p_t^{u,N}(y_2)|\leq C|y_1-y_2|^\beta t^{-\frac{\beta}{2}} \sum_{i=1}^2 \int_{\R^d} g(4t,x-y_i)\nu_0^u(dx)$, for all $y_1,y_2 \in \R^d$;
     \item $|p^{u,N}_{t_1}(y)-p^{u,N}_{t_2}(y)|\leq C|t_1-t_2|^{\beta/2} \sum_{j=1}^2 t_j^{-\frac{\beta}{2}} \int_{\R^d} g(2t_j,x-y)\nu_0^u(dx)$, for all $t_1,t_2\in (0,T)$.
 \end{enumerate}
 \begin{lemma}\label{lemma.timemarginals.approximate.regularity} The following hold:
\begin{enumerate}[label=(\roman*)]
    \item[(i)] $(t,y,u) \mapsto p_t^{u,N}(y) \in L^p((0,T]\times\R^d\times[0,1])$, for all $p<\frac{d+2}{d}$;
    \item [(ii)] $(t,y) \mapsto (p_t^{u,N}(y))_{u\in [0,1]} \in C((0,T]\times\R^d;L^p([0,1])$, for all $p\in [1,\infty)$.
\end{enumerate}
\end{lemma}
\begin{remark}
    Note that due to Lemma~\ref{lemma.heatkernelestimate.2} (i), for all $t \in (0,T], y \in \R^d$: \\$[u\mapsto p_t^{u,N}(y)]\in L^\infty([0,1])$.
\end{remark}
\begin{proof}[Proof of Lemma~\ref{lemma.timemarginals.approximate.regularity}]
    Regarding (i): First we show that $(0,T]\times[0,1]\times\R^d\ni(t,u,y) \mapsto p_t^{u,N}(y)$ is Borel measurable.
Since by Lemma~\ref{lemma.heatkernelestimate.2} (ii), $y\mapsto p^{u,N}_t(y)$ is continuous for each $u\in[0,1],t\in (0,T],N \in \N$, it is sufficient to prove that the map $(t,u)\mapsto \int f(y)p^{u,N}_t(y)dy$ is Borel measurable, for each bounded Borel measurable function $f:\R^d\to\R$. Clearly, it is enough to show that for all $i \in \N$, $(t,u)\mapsto \int f(y)p^{u,N}_t(y)dy$ is Borel measurable when restricted to the interval $(ih,(i+1)h]\times[0,1]$.
Since, $X^u_0$ and $W$ are independent, it is straight-forward to see that the assertion is true for $i=0$.
For $i\in\N$, $t\in (ih,(i+1)h]$, it follows from \eqref{EulerScheme.2} and the independence of $X^{N}(ih)$ and $(W(t)-W(ih))$ that
\begin{align}\label{lemma.timemarginals.approximate.regularity:1}
    &\int_{\R^d} f(y) p^{u,N}_t(y)dy \\
    &= \int_{\R^d}\int_{\R^d} f\left(y + \int_{ih}^t b\left(s,y,\langle \tilde{K}_1(u), p^{\bullet,N}_{ih}(y)\rangle\right) ds +z\right) p_{ih}^{u,N}(y)g\left(t-ih,z\right)\ dy dz.\notag
\end{align}
Due to the assumptions on $\tilde{K}_1$, the desired measurability can now be deduced directly from the right-hand side of \eqref{lemma.timemarginals.approximate.regularity:1}.
Let us now prove the asserted integrability properties.
Since $\nu_0^u\in \sP(\R^d)$, Lemma~\ref{lemma.heatkernelestimate.2} and Minkowski's inequality imply that
\begin{align}
    \int_0^T \int_{[0,1]} \int_{\R^d} |p_t^{u,N}(y)|^p dydudt
   & \lesssim \int_0^T\int_0^1 \int_{\R^d} (g(4t,\cdot)\ast \nu_0^u)^pdydudt\notag\\
    &\lesssim \int_0^T\|g(4t,\cdot)\|^p_{L^p}dt
    \lesssim \int_0^T t^{(1-p)d/2}dt<\infty.
\end{align}
Regarding (ii): This is an easy consequence of Lemma~\ref{lemma.heatkernelestimate.2} (ii) and (iii).
This finishes the proof.
\end{proof}
\end{lemma}

\section{Proof of Theorem~\ref{theorem.GMVSDE.weakExistence}}\label{section.proof.GMVSDE.weakExistence}
This proof uses partly the line of reasoning in \cite[Proof of Theorem 1.2]{hao2021euler}. However, because of the difficulty to work with an uncountable system of MVSDE in the form of \eqref{GMVSDE} and the induced difficulties arising from working with Fubini extensions, the proof requires substantial further development.

Assume that (H1) and (H2) hold. We consider the Euler scheme \eqref{EulerScheme.1}-\eqref{EulerScheme.2} with initial distribution $\sL_{X^u_0}(dx)=\nu_0^u(dx)$, $u \in [0,1]$.
    In order to find limits of the laws of $X^{u,N}(t)$, $u \in [0,1]$, as $N\to \infty$, we prove the following relative compactness result for their one-dimensional time marginal law densities $(p^{u,N}_\cdot(\cdot))_{u\in [0,1]}$ in $C((0,T]\times\R^d;L^1([0,1]))$ with respect to the topology induced by local uniform convergence.
    
\begin{lemma}\label{lemma.ArzelaAscoli}
    There exists a subsequence $(N_k)_{k\in\N}$ and ${(p^u)_{u\in [0,1]}} \in {C((0,T]\times \R^d; L^1([0,1]))}$ such that for every $1/T < M \in \N$,
    \begin{align}\label{lemma.ArzelaAscoli:1}
        \lim_{k\to \infty}\sup_{|y|\leq M}\sup_{1/M\leq t\leq T}\|p_t^{\cdot,N_k}(y)-p_t^{\cdot}(y)\|_{L^1([0,1])}=0.
    \end{align}
\end{lemma}
\begin{proof}
    Let $M>1/T$. Lemma~\ref{lemma.heatkernelestimate.2} (i) implies that
    \begin{align}
       \sup_{N\in \N}\sup_{y\in \R^d}\sup_{1/M< t\leq T}\|p_t^{\cdot,N}(y)\|_{L^1([0,1])} <\infty.
    \end{align}
    Furthermore, for all $\beta\in (0,1), t_1,t_2 \in [1/M,T]$ and $y_1,y_2 \in \R^d$,
    \begin{align}
        \|p_{t_1}^{\cdot,N}(y_1)&-p_{t_2}^{\cdot,N}(y_2)\|_{L^1([0,1])}\notag\\
        &\leq\|p_{t_1}^{\cdot,N}(y_1)-p_{t_2}^{\cdot,N}(y_1)\|_{L^1([0,1])}
        +\|p_{t_2}^{\cdot,N}(y_1)-p_{t_2}^{\cdot,N}(y_2)\|_{L^1([0,1])}\notag\\
        &\lesssim |t_1-t_2|^{\beta/2} \sum_{j=1}^2 t_j^{-\frac{\beta}{2}} \int_0^1 \left|\int_{\R^d} g(2t_j,x-y_1)\nu^u_0(dx)\right|du\notag\\
        &\quad +|y_1-y_2|^\beta t_2^{-\frac{\beta}{2}} \sum_{i=1}^2 \int_0^1 \left| \int_{\R^d} g(4t_2,x-y_i)\nu^u_0(dx)\right|du\notag\\
        &\lesssim_M (|t_1-t_2|^{\beta/2}+ |y_1-y_2|^\beta)\label{lemma.ArzelaAscoli:2},
    \end{align}
    where we used Lemma~\ref{lemma.heatkernelestimate.2} (ii) and (iii) in the second estimate.
    Combining the previous estimates with Claim \ref{claim.ArzelaAscoli.compactnessInL1} below, the assertion of the current lemma follows from the Arzel\`a--Ascoli theorem.
To avoid breaking the continuity of the argument, we postpone its proof to Section \ref{section.proof.claim.ArzelaAscoli.compactnessInL1} below.   
    \begin{claim}\label{claim.ArzelaAscoli.compactnessInL1}
    	For all $t\in (0,T]$ and $y\in \R^d$,
    \begin{align}\label{claim.ArzelaAscoli.compactnessInL1:1}
        \sup_{N\in\N}\|p_t^{u+\kappa,N}(y)-p_t^{u,N}(y)\|_{L^1([0,1])} \to 0, \text{ as } |\kappa|\to 0,
    \end{align}
    where we set $p^{u+\kappa,N}\equiv p^u$, if $u+\kappa \notin [0,1]$.
    \end{claim}

\end{proof}

Since $\|b^{N,u}\|_{L^\infty}\leq \|b\|_{L^\infty}$, the following lemma holds. As its proof is straightforward, we omit it here.
\begin{lemma}\label{lemma.Kolmogorov}
    For every $T>0$, there is a constant $C>0$ such that, for all $s,t \in [0,T]$,
    \begin{align}
        \sup_{u\in [0,1]} \sup_{N\in \N} \E |X^{u,N}(t)-X^{u,N}(s)|^4 \leq C |t-s|^2.
    \end{align}
\end{lemma}
By $\mathbb{Q}_N^u$ we denote the law of $(X^{u,N}, W)$ in $\mathcal{C}_T^d\times\mathcal{C}_{0,T}^d$, for $u \in [0,1]$ and $N\in\N$.
 From now on, we consider $\mathbb{Q}_N^u$ for the subsequence $\{N_k\}_{k\in \N}\subset\N$ which is provided by Lemma~\ref{lemma.ArzelaAscoli}. For the sake of better readability, we do not relabel the original sequence.
 
Fix an arbitrary $u \in [0,1]$.
By Lemma~\ref{lemma.Kolmogorov} and Kolmogorov's criterion, $(\mathbb{Q}_N^u)_{N\in\N}$ is tight.
We conclude that, by Prokhorov's theorem, there is a subsequence (depending on $u$ and not denoted differently) and a probability measure $\mathbb{Q}^u\in \sP(\mathcal{C}_T^d)$ such that
\begin{align}
    \mathbb{Q}_N^u\to \mathbb{Q}^u\ \text{ in $\mathcal{P}(\mathcal{C}_T^d\times\mathcal{C}_{0,T}^d)$,\quad as  $N\to \infty$}.
\end{align}
Via Skorokhod's representation theorem, we find a probability space $(\tilde{\Omega}^u, \tilde{\sF}^u, \tilde{\P}^u)$ and random variables $(\tilde{X}^{u,N}, \tilde{W}^{u,N})$ and $(\tilde{X}^u, \tilde{W}^u)$ thereon such that
\begin{align}\label{proof.theorem.weakExistence:convergence.skorokhod}
    \lim_{N\to\infty}(\tilde{X}^{u,N}, \tilde{W}^{u,N}) = (\tilde{X}^u, \tilde{W}^u) \text{ \ \ } \tilde{\P}^u\text{-a.s.}
\end{align}
and
\begin{align}
    \mathbb{Q}_N^u=\tilde{\P}^u\circ (\tilde{X}^{u,N}, \tilde{W}^{u,N})^{-1},\quad \mathbb{Q}^u=\tilde{\P}^u\circ (\tilde{X}^u, \tilde{W}^u)^{-1}.
\end{align}
In particular, $\sL_{\tilde{X}^{u,N}(t)}(dy) = p_t^{u,N}(y)dy$, for all $t\in [0,T]$. Furthermore, note that $\tilde{W}^{u,N}$ is a ${\sigma(\tilde{W}^{u,N}(s),\tilde{X}^{u,N}(s) : s\leq t)}$-Wiener process, and $(\tilde{X}^{u,N}, \tilde{W}^{u,N})$ satisfies \eqref{GMVSDE.EulerScheme.1} when replacing $(X^{u,N}, W^{u,N})$.

Next, we prove that $(\tilde{X}^u, \tilde{W}^u)$ is a (classical) weak solution to 
\begin{align}\label{proof.theorem.weakExistence:GMVSDE.fixed.u}
    dX(t) 
      =
    b\Big(t, X(t),  \langle \tilde{K}_1(u), p^{\bullet}_t(X(t))\rangle  \Big) dt + \sqrt{2}dW(t),\quad t\in [0,T],
\end{align}
for $du$-a.e. $u\in[0,1]$.
To this end, it is sufficient to prove that there exists a $du$-null set $\mathcal{N} \in \sB ([0,1])$ such that for all $u \in \mathcal{N}^\complement$ there exists another subsequence $\{N\}\subset \N$ (not denoted differently) such that, as $N\to \infty$,
\begin{align}\label{proof.theorem.weakExistence:identification}
	\tilde{\E}^u&\int_0^T \left|b^{u,N}(s,\tilde{X}^{u,N}(\phi_N(s)))-b(s,\tilde{X}^u(s),\langle \tilde{K}_1(u),p^\bullet_s(\tilde{X}^u(s))\rangle) ds\right| du\longrightarrow 0.
\end{align}
Here, $\tilde{\E}^u$ denotes the expectation with respect to the probability measure $\tilde{\P}^u$.
It is easy to see that \eqref{proof.theorem.weakExistence:identification} is implied by (i) $\wedge$ (ii), where
\begin{enumerate}[label=(\roman*)]
	\item as $N\to \infty$,
		\begin{align*}
			\int_0^1 {\E}\int_h^T \Big| b&(s,{X}^{u,N}(\phi_N(s)),\langle\tilde{K}_1(u),p^{\bullet,N}_{\phi_N(s)}({X}^{u,N}(\phi_N(s)))\rangle)\\
			&-b(s,{X}^{u,N}(\phi_N(s)),\langle \tilde{K}_1(u),p^\bullet_s({X}^{u,N}(\phi_N(s)))\rangle) \Big|dsdu \longrightarrow 0;
        \end{align*}
        
	\item for all $u \in [0,1]$,
    as $N\to \infty$,
		\begin{align*}
			&\tilde{\E}^u \int_h^T \Big|b(s,\tilde{X}^{u,N}(\phi_N(s)),\langle\tilde{K}_1(u),p^{\bullet}_{s}(\tilde{X}^{u,N}(\phi_N(s)))\rangle)\\
            &\quad \quad \quad \quad -b(s,\tilde{X}^u(s),\langle \tilde{K}_1(u),p^\bullet_s(\tilde{X}^u(s))\rangle) \Big|ds\longrightarrow 0.
        \end{align*}
\end{enumerate}
The convergence in (ii) is proved analogously to \cite[Proof of Theorem 1.2 starting in p. 1010. l.3]{hao2021euler} (which is possible due to Lemma~\ref{lemma.heatkernelestimate.2} (i)).

Now we prove (i). Let $s\in (0,T)$ and $R>0$.
Then, by (H1), (H2), \eqref{theorem.weakExistence:condition.b}, and H\"older's inequality, we have
\begin{align}\label{proof.theorem.weakExistence:inequalities}
	\notag&\int_0^1 {\E} 
	\Big| 
	b(s,{X}^{u,N}(\phi_N(s)),\langle\tilde{K}_1(u),p^{\bullet,N}_{\phi_N(s)}({X}^{u,N}(\phi_N(s)))\rangle)\\
	\notag &\qquad \ \ -b(s,{X}^{u,N}(\phi_N(s)),\langle \tilde{K}_1(u),p^\bullet_s({X}^{u,N}(\phi_N(s)))\rangle)
	\Big|du\\ \notag
	&\lesssim\int_0^1 {\E}\ 1_{B_R(0)}({X}^{u,N}(\phi_N(s))) \\\notag
    &\quad \quad \quad \quad \quad \times\Big |\langle\tilde{K}_1(u),p^{\bullet,N}_{\phi_N(s)}({X}^{u,N}(\phi_N(s)))-p^\bullet_s({X}^{u,N}(\phi_N(s)))\rangle \Big|^\alpha du\\\notag
	&\quad \quad +\|b\|_{L^\infty(\R^d)}\int_0^1 \P(|{X}^{u,N}(\phi_N(s))|\geq R)du\\\notag
	&\lesssim \int_0^1\int_{B_R(0)}\left |\langle\tilde{K}_1(u),p^{\bullet,N}_{\phi_N(s)}(x)-p^\bullet_s(x)\rangle \right|^\alpha p_{\phi_N(s)}^{u,N}(x) dxdu \\ \notag
	& \quad \quad +\int_0^1\P(|X^{u,N}(\phi_N(s))|\geq R)du \\\notag
	&\lesssim \|v^\cdot_0\|_{L^\infty} R^{d(1-\alpha)} \left(\int_0^1\int_{B_R(0)}\left |\langle\tilde{K}_1(u),p^{\bullet,N}_{\phi_N(s)}(x)-p^\bullet_s(x)\rangle \right| dxdu\right)^\alpha\\\notag
	&\quad \quad +\int_0^1\P(|X^{u}(0)|+s\|b\|_{L^\infty}+ \sqrt{2}|W(\phi_N(s))|\geq R)du\\\notag
	&\lesssim R^{d(1-\alpha)} \left(\int_{B_R(0)}\left \|p^{\bullet,N}_{\phi_N(s)}(x)-p^\bullet_s(x)\right\|_{L^1([0,1])} dx\right)^\alpha\\
	&\quad \quad +\int_0^1\P(|X^u(0)|+s\|b\|_{L^\infty}+ \sqrt{2}|W(\phi_N(s))|\geq R)du.
\end{align}
Note that the first summand on the right-hand side converges to zero, as $N\to \infty$, by \eqref{lemma.ArzelaAscoli:1}, \eqref{lemma.ArzelaAscoli:2}, and Lebesgue's dominated convergence theorem in combination with (H1) and Lemma~\ref{lemma.heatkernelestimate.2} (i).
Regarding the second summand, we observe that, by Chebyshev's inequality,
\begin{align*}
	\P(|X^u(0)|+s\|b\|_{L^\infty}+& \sqrt{2}|W(\phi_N(s))|\geq R)\\ 
	&\leq \P(|X^u(0)|+s\|b\|_{L^\infty}\geq R/2) + \P(\sqrt{2}|W(\phi_N(s))|\geq R/2)\\
	&\leq \nu_0^u(|\cdot|+s\|b\|_{L^\infty}\geq R/2) + \frac 8 {R^2}s.
	\end{align*}
Hence, by the monotone convergence theorem, we obtain
\begin{align*}
	\int_0^1\P(|X^u(0)|+s\|b\|_{L^\infty}+ \sqrt{2}|W(\phi_N(s))|\geq R)du \to 0, \text{ as } R\to \infty,
\end{align*}
uniformly in $N$.
Therefore, integrating over $(h,T)$ in \eqref{proof.theorem.weakExistence:inequalities}, Lebesgue's dominated convergence theorem yields (i).

Now let  $u\in \mathcal{N}^\complement$ such that \eqref{proof.theorem.weakExistence:identification} holds. Note that the SDE \eqref{proof.theorem.weakExistence:GMVSDE.fixed.u} (w.r.t. $u$) has a unique strong solution, see \cite[Theorem 1]{veretennikov1980strong}. This makes $\mathbb{Q}^u$ the unique accumulation point of $(\mathbb{Q}^u_N)$ for the initially considered sequence provided by Lemma~\ref{lemma.ArzelaAscoli}. Since Lemma~\ref{lemma.Kolmogorov} guarantees tightness of every subsequence of $(\mathbb{Q}^u_N)$, it is now standard to conclude that $\mathbb{Q}^u_N$ converges weakly to $\mathbb{Q}^u$ for the whole sequence. Note that this sequence is independent of $u$.
Note that the map $[0,1] \ni u\mapsto \mathbb{Q}^u_N \in \mathcal{P}(\mathcal{C}_T^d\times\mathcal{C}_{0,T}^d)$ is Borel measurable, which is due to the construction of $X^{u,N}$ in terms of the Euler scheme. Hence, $[0,1] \ni u\mapsto \mathbb{Q}^u \in \mathcal{P}(\mathcal{C}_T^d\times\mathcal{C}_{0,T}^d)$ has a Borel measurable version.

Furthermore, by \eqref{proof.theorem.weakExistence:convergence.skorokhod}
and \eqref{lemma.ArzelaAscoli:1}, we obtain,
for every $\varphi \in C_c(\R^d)$, $\psi \in C([0,1])$, and every $t\in (0,T]$,
\begin{align*}
    \int_0^1 \psi(u)\tilde{\E}^u \varphi(\tilde{X}^u(t))du
    &=\lim_{N\to \infty}\int_0^1 \psi(u)\tilde{\E}^u \varphi(\tilde{X}^{u,N}(t))du\\
    &= \lim_{N\to \infty}\int_0^1 \psi(u)\int_{\R^d}\varphi(y)p^{u,N}_t(y)dydu\\
    &=\int_0^1 \psi(u)\int_{\R^d} \varphi(y)p^u_t(y)dydu.
\end{align*}
Hence, we conclude that
	\begin{align*}
		\sL_{\tilde{X}^u(t)}(dy)du=p^{u}_t(y)dydu\quad \forall t\in (0,T].
	\end{align*}

Now, by setting $\varphi(u):=\sL_{(\tilde{X}^u,\tilde{W}^u)}$ in Theorem~\ref{appendix.fubiniextension.theorem.existence}, we obtain
a probability space $([0,1],\bar{\mathcal I},\bar{\lambda})$ extending the
standard Lebesgue space, a probability space $(\bar{\Omega},\bar{\sF},\bar{\sP})$, and a Fubini extension $([0,1]\times\bar{\Omega}, \bar{\mathcal{I}}\boxtimes\bar{\sF},\bar{\lambda}\boxtimes\bar{\P})$ which supports a measurable random element $(\bar{X}, \bar{W}):[0,1]\times\bar{\Omega}\to \mathcal{C}_T^d\times\mathcal{C}_{0,T}^d$ so that,  $(\bar{X}^u, \bar{W}^u)$ is e.p.i. in the parameter $u$, and its distribution is equal to $\mathbb{Q}^{u}$, for $du$-a.e. $u \in [0,1]$. It is straightforward to see that, for $du$-a.e. $u \in [0,1]$, $\bar{W}^u$ is an $(\bar{\sF}_t^u)$-Wiener process, where
		$$\bar{\sF}_t^u:=\bigcap_{\varepsilon>0}\sigma\left(\sigma(\bar{W}^u(s),\bar{X}^u(s); s\leq t+\varepsilon),\mathcal{N}\right),\quad \mathcal{N}:=\{N \in \tilde{\sF}: \bar{\P}(N)=0\},$$
and that $(\bar{X}^u, \bar{W}^u)$ is a weak solution to equation \eqref{proof.theorem.weakExistence:GMVSDE.fixed.u}.
Now, it is easy to conclude that $(\bar{X},\bar{W})$ is a weak solution to \eqref{GMVSDE} according to Definition \ref{definition.GMVSDE.weakSolution} with respect to the filtrations $(\bar{\sF}_\cdot^u)_{u\in [0,1]}$.

The last part of the statement follows directly by an easy application of It\^o's formula to the constructed weak solution.

This concludes the proof of Theorem~\ref{theorem.GMVSDE.weakExistence}.\qed

\section{Details for the proof of Theorem~\ref{theorem.GMVSDE.weakExistence}: Proof of Claim \ref{claim.ArzelaAscoli.compactnessInL1}}\label{section.proof.claim.ArzelaAscoli.compactnessInL1}
The proof of the claim is divided into three steps, each proved in a separate subsection.

    \vspace{0.5em}
    \noindent\textbf{Step 1:} We show that, for every $\beta\in (0,1/2)$, for every $t \in (0,T]$ with $t=kh=k[T/N]$, for some $k\in \N$, and every $y \in \R^d$,
    \begin{align}\label{proof.claim.ArzelaAscoli.compactnessInL1:step1}
    	|&p_t^{u+\kappa,N}(y)-p_t^{u,N}(y)| \notag\\
    	&\lesssim_{t,\beta}
		\int_{\R^d}|v_0^{u+\kappa}(y)-v_0^u(y)|dy
		+\|\tilde{K}_1(u+\kappa)-\tilde{K}_1(u)\|_{\mathrm{TV}}^\alpha
		+ \|\tilde{K}_1(u+\kappa)-\tilde{K}_1(u)\|_{\mathrm{TV}}^{\alpha\beta}.
    \end{align}
    \vspace{0.5em}
\noindent\textbf{Step 2:} We show that, for every $\beta\in (0,1/2)$, for every $t \in (0,T]$, and every $y \in \R^d$,
    \begin{align*}
    	|p_t^{u+\kappa,N}(y)-p_t^{u,N}(y)|
    	&\lesssim_{t,\beta}
		\int_{\R^d}|v_0^{u+\kappa}(y)-v_0^u(y)|dy
		+\|\tilde{K}_1(u+\kappa)-\tilde{K}_1(u)\|_{\mathrm{TV}}^\alpha\\
		&\quad \quad \quad + \|\tilde{K}_1(u+\kappa)-\tilde{K}_1(u)\|_{\mathrm{TV}}^{\alpha\beta}.
    \end{align*}
    \vspace{0.5em}
\noindent\textbf{Step 3:} We show that \eqref{claim.ArzelaAscoli.compactnessInL1:1} holds.
    \subsection{Proof of Step 1:}
    By Lemma~\ref{lemma.DuhamelsFormula} and Remark~\ref{remark.DuhamelsFormula},
    we have
    \begin{align*}
        p_t^{u+\kappa,N}&(y)-p_t^{u,N}(y)
        =\int_{\R^d} p^{u+\kappa,N,x}_t(y)\nu_0^{u+\kappa}(dx)-\int_{\R^d} p^{u,N,x}_t(y)\nu_0^{u}(dx)\\
        &= \int g(t,x-y)(\nu_0^{u+\kappa}-\nu_0^{u})(dx)\\
        &\quad+ \int_0^t \E b^{u+\kappa,N}(s, X^{u+\kappa,N}(\phi_N(s)))\cdot \nabla g(t-s,y-X^{u+\kappa,N}(s))ds\\
        &\quad - \int_0^t\E b^{u,N}(s, X^{u,N}(\phi_N(s)))\cdot \nabla g(t-s,y-X^{u,N}(s))ds.
    \end{align*}
    
First, we focus on estimating the sum of the second and third summand on the right-hand side.

    \begin{align*}
        &\Big| \int_0^t \E b^{u+\kappa,N}(s, X^{u+\kappa,N}(\phi_N(s)))\cdot \nabla g(t-s,y-X^{u+\kappa,N}(s))\\
        &\quad\quad\quad\quad\quad\ - \E b^{u,N}(s, X^{u,N}(\phi_N(s)))\cdot \nabla g(t-s,y-X^{u,N}(s))ds\Big|\\
        &\lesssim \sum_{i=1}^{k-1}\int_{ih}^{(i+1)h} \Big|\E b(s,X^{u+\kappa,N}(ih),\langle \tilde{K}_1(u+\kappa),p^{\bullet,N}_{ih}(X^{u+\kappa,N}(ih))\rangle)\\
        &\quad \quad \quad \quad \quad \quad \quad \quad \cdot \nabla g(t-s,y-X^{u+\kappa,N}(s))\\
        &\quad\quad\quad\quad\quad\quad\quad -\E b(s,X^{u,N}(ih),\langle \tilde{K}_1(u),p^{\bullet,N}_{ih}(X^{u,N}(ih))\rangle)\\
        &\quad\quad\quad\quad\quad\quad\quad\quad\quad\cdot \nabla g(t-s,y-X^{u,N}(s))\Big|ds.
    \end{align*}
        Let $s\in (ih,(i+1)h)$, $i \in \N$. Recall that
    \begin{align*}
        X^{u,N}(s) &= X^{u,N}(ih) + \int_{ih}^s b(r,X^{u,N}(ih),\langle \tilde{K}_1(u),p^{\bullet,N}_{ih}(X^{u,N}(ih))\rangle)dr + \sqrt{2}(W(s)-W(ih)).\\
        \intertext{Hence, the independence of $(W(s)-W(ih))$ and $X^{u}(ih)$ yields}
            \E &  b(s,X^{u,N}(ih),\langle \tilde{K}_1(u),p^{\bullet,N}_{ih}(X^{u,N}(ih))\rangle)\cdot \nabla g(t-s,y-X^{u,N}(s))\\
        = & \int_{\R^d}\int_{\R^d}
        b(s,x,\langle \tilde{K}_1(u),p^{\bullet,N}_{ih}(x)\rangle)\\
        &\quad \quad \quad \quad \cdot \nabla g\left(t-s,y-\left(x+\int_{ih}^s b(r,x, \langle \tilde{K}_1(u),p^{\bullet,N}_{ih}(x)\rangle)dr + w\right)\right)\\
        &\quad \quad \quad \times p^{u,N}_{ih}(x)g(s-ih,w)dwdx
        \end{align*}
    To proceed, we express the following difference as a telescoping sum:
    \begin{align*}
    &\E b(s,X^{u+\kappa,N}(ih),\langle \tilde{K}_1(u+\kappa),p^{\bullet,N}_{ih}(X^{u+\kappa,N}(ih))\rangle)\cdot \nabla g(t-s,y-X^{u+\kappa,N}(s))\\
        &-\E b(s,X^{u,N}(ih),\langle \tilde{K}_1(u),p^{\bullet,N}_{ih}(X^{u,N}(ih))\rangle)\cdot \nabla g(t-s,y-X^{u,N}(s))\\
        &=\int_{\R^d}\int_{\R^d}
        \left[b(s,x,\langle \tilde{K}_1(u+\kappa),p^{\bullet,N}_{ih}(x)\rangle)-b(s,x,\langle \tilde{K}_1(u),p^{\bullet,N}_{ih}(x)\rangle)\right]\\
        &\quad \quad \quad \quad \quad \quad \cdot \nabla g\left(t-s,y-\left(x+\int_{ih}^s b(r,x, \langle \tilde{K}_1(u+\kappa),p^{\bullet,N}_{ih}(x)\rangle)dr + w\right)\right)\\
        &\quad \quad \quad \quad \times p^{u+\kappa,N}_{ih}(x)g(s-ih,w)dwdx\\
        &\quad + \int_{\R^d}\int_{\R^d}
        b(s,x,\langle \tilde{K}_1(u),p^{\bullet,N}_{ih}(x)\rangle)\\
        &\quad \quad \quad \quad \quad \quad  \cdot 
        \Bigg[\nabla  g\left(t-s,y-\left(x+\int_{ih}^s b(r,x, \langle \tilde{K}_1(u+\kappa),p^{\bullet,N}_{ih}(x)\rangle)dr + w\right)\right) \\
        &\quad \quad \quad \quad \quad \quad \quad -\nabla g\left(t-s,y-\left(x+\int_{ih}^s b(r,x,\langle \tilde{K}_1(u),p^{\bullet,N}_{ih}(x)\rangle)dr + w\right)\right)\Bigg]\\
        &\quad \quad \quad \quad \times p^{u+\kappa,N}_{ih}(x)g(s-ih,w)dwdx\\
        &\quad + \int_{\R^d}\int_{\R^d}
        b(s,x,\langle \tilde{K}_1(u),p^{\bullet,N}_{ih}(x)\rangle)\\
        &\quad \quad\quad\quad\quad\quad \cdot \nabla g\left(t-s,y-\left(x+\int_{ih}^s b(r,x, \langle \tilde{K}_1(u),p^{\bullet,N}_{ih}(x)\rangle)dr + w\right)\right)\\
        &\quad \quad \quad \quad \quad \times\left[p^{u+\kappa,N}_{ih}(x)-p^{u,N}_{ih}(x)\right]g(s-ih,w)dwdx\\
        &\quad =: I^1(N,i,\kappa,s,t) + I^2(N,i,\kappa,s,t) + I^3(N,i,\kappa,s,t)
        =: I^1 + I^2 + I^3.
    \end{align*}
   We continue by estimating each of these summands individually.
   
    \vspace{1em}
    \noindent\textbf{Regarding $\bm{I^1}$}:
    Due to \eqref{theorem.weakExistence:condition.b}, we have, for all $x\in \R^d$,
    \begin{align*}
        &|b(s,x,\langle \tilde{K}_1(u+\kappa),p^{\bullet,N}_{ih}(x)\rangle)-b(s,x,\langle \tilde{K}_1(u),p^{\bullet,N}_{ih}(x)\rangle)|\\
        &\quad \lesssim |\langle \tilde{K}_1(u+\kappa)-\tilde{K}_1(u),p^{\bullet,N}_{ih}(x)\rangle|^\alpha.
    \end{align*}
   By (H1) and Lemma~\ref{lemma.heatkernelestimate.2} (i), there exists a constant $C\in(0,\infty)$ independent of the variables involved, such that for arbitrary $\bar{t}\in (0,T]$ and $\bar{v}\in [0,1]$
    \begin{align}\label{proof.claim.ArzelaAscoli.compactnessInL1:boundedDensities}
        p_{\bar{t}}^{\bar{v},N}(y)\lesssim \int_{\R^d}g(4\bar{t},y-z)v_0^{\bar{v}}(z)dz \leq C.
    \end{align}
    Hence, by Lemma~\ref{appendix.lemma.hao2021euler.propertiesOfFundamentalSolution.1} (iii), we have
    \begin{align}
        |I^1|
        &\lesssim \int_{\R^d}\int_{\R^d}
        \|\tilde{K}_1(u+\kappa)-\tilde{K}_1(u)\|_{\mathrm{TV}}^\alpha \notag\\
      &\quad \quad \quad \quad \times \left|\nabla g\left(t-s,y-\left(x+\int_{ih}^s b(r,x, \langle \tilde{K}_1(u+\kappa),p^{\bullet,N}_{ih}(x)\rangle)dr + w\right)\right)\right|\notag \\
      &\quad \quad \quad \times p^{u+\kappa,N}_{ih}(x)g(s-ih,w)dwdx\notag \\
        &\lesssim  \frac1 {\sqrt{t-s}}\int_{\R^d}\int_{\R^d}
        \|\tilde{K}_1(u+\kappa)-\tilde{K}_1(u)\|_{\mathrm{TV}}^\alpha \notag\\
     &\quad \quad \quad \quad \times g\left(2(t-s),y-\left(x+\int_{ih}^s b(r,x, \langle \tilde{K}_1(u+\kappa),p^{\bullet,N}_{ih}(x)\rangle)dr + w\right)\right)\notag\\
     &\quad \quad \quad \times p_{ih}^{u+\kappa,N}(x)g(s-ih,w)dwdx \label{proof.claim.ArzelaAscoli.compactnessInL1:inequalities}
    \end{align}
    Note that by Lemma~\ref{appendix.lemma.hao2021euler.propertiesOfFundamentalSolution.1} (i) and the elementary fact that $g(t,x)\leq 2^{d/2}g(2t,x)$, we have    \begin{align*}
        &\int g\left(2(t-s), y-\left(x+\int_{ih}^s b(r,x, \langle \tilde{K}_1(u+\kappa),p^{\bullet,N}_{ih}(x)\rangle)dr + w\right)\right) g(s-ih,w)dw\\
        &\leq 2^{d/2} g\left(2(t-ih), y-\left(x+\int_{ih}^s b(r,x, \langle \tilde{K}_1(u+\kappa),p^{\bullet,N}_{ih}(x)\rangle)dr\right)\right).
    \end{align*}
    Hence, by Lemma~\ref{appendix.lemma.hao2021euler.propertiesOfFundamentalSolution.1} (ii) and \eqref{proof.claim.ArzelaAscoli.compactnessInL1:boundedDensities}, we may further estimate the right-hand side of \eqref{proof.claim.ArzelaAscoli.compactnessInL1:inequalities} as follows.
    \begin{align*}
    |I^1|
      &\lesssim \frac1 {\sqrt{t-s}}
        \|\tilde{K}_1(u+\kappa)-\tilde{K}_1(u)\|_{\mathrm{TV}}^\alpha\\
     &\quad \quad \times \int_{\R^d}g\left(2(t-ih),y-\left(x+\int_{ih}^s b(r,x, \langle \tilde{K}_1(u+\kappa),p^{\bullet,N}_{ih}(x)\rangle)dr\right)\right) dx\\
     &\lesssim \frac{e^{(t-ih)\|b\|_{L^\infty}^2}} {\sqrt{t-s}}
        \|\tilde{K}_1(u+\kappa)-\tilde{K}_1(u)\|_{\mathrm{TV}}^\alpha.
    \end{align*}
    
    \vspace{1em}
    \noindent\textbf{Regarding $\bm{I^2}$:}
    By Lemma~\ref{appendix.lemma.hao2021euler.propertiesOfFundamentalSolution.2}, \eqref{proof.claim.ArzelaAscoli.compactnessInL1:boundedDensities}, and Lemma~\ref{appendix.lemma.hao2021euler.propertiesOfFundamentalSolution.1} (i) and (ii), for an arbitrarily fixed $\beta \in (0,1)$, there exists a corresponding constant $C_\beta>0$ such that
    \begin{align*}
        |I^2|&= \Bigg|\int_{\R^d}\int_{\R^d}
        b(s,x,\langle \tilde{K}_1(u),p^{\bullet,N}_{ih}(x)\rangle)\\
        &\quad \quad\quad \cdot 
        \Bigg[\nabla  g\left(t-s,y-\left(x+\int_{ih}^s b(r,x, \langle \tilde{K}_1(u+\kappa),p^{\bullet,N}_{ih}(x)\rangle)dr + w\right)\right) \\
        &\quad \quad \quad \quad 
        \begin{aligned}[t]
	&\quad -\nabla g\left(t-s,y-\left(x+\int_{ih}^s b(r,x,\langle \tilde{K}_1(u),p^{\bullet,N}_{ih}(x)\rangle)dr + w\right)\right)\Bigg]\\
        &\times p^{u+\kappa,N}_{ih}(x)g(s-ih,w)dwdx\Bigg|
\end{aligned}\\
        &\lesssim\frac{ C_\beta \|b\|_{L^\infty} \|v_0^\cdot\|_{L^\infty}}{ (t-s)^{1/2+\beta}}\\
        &\quad\times \int_{\R^d}\int_{\R^d} \Bigg|\int_{ih}^s b(r,x, \langle \tilde{K}_1(u+\kappa),p^{\bullet,N}_{ih}(x)\rangle)- b(r,x, \langle \tilde{K}_1(u),p^{\bullet,N}_{ih}(x)\rangle)dr\Bigg|^\beta\\
        &\quad \quad \quad \quad \times \Bigg[g\left(4(t-s), y-\left(x+\int_{ih}^s b(r,x, \langle \tilde{K}_1(u+\kappa),p^{\bullet,N}_{ih}(x)\rangle)dr + w\right)\right)\\
        &\quad\quad\quad\quad\quad\quad +g\left(4(t-s), y-\left(x+\int_{ih}^s b(r,x,\langle \tilde{K}_1(u),p^{\bullet,N}_{ih}(x)\rangle)dr + w\right) \right)\Bigg]\\
        &\quad\quad \quad \times g(s-ih,w)dwdx\\
        &\lesssim \frac{ 2^{d+1}C_\beta \|b\|_{L^\infty} \|v_0^\cdot\|_{L^\infty}^{1+\alpha\beta}e^{(t-ih)\|b\|_{L^\infty}^2}h^\beta}{ (t-s)^{1/2+\beta}}\|\tilde{K}_1(u+\kappa)-\tilde{K}_1(u)\|_{\mathrm{TV}}^{\alpha\beta}\\
        &\lesssim \frac{h^\beta}{ (t-s)^{1/2+\beta}}\|\tilde{K}_1(u+\kappa)-\tilde{K}_1(u)\|_{\mathrm{TV}}^{\alpha\beta}.
    \end{align*}
    
    \vspace{1em}
    \noindent\textbf{Regarding $\bm{I^3}$:} We have
    \begin{align*}
        &|I^3|= \Bigg|\int_{\R^d}\int_{\R^d}
        \begin{aligned}[t]
        &b(s,x,\langle \tilde{K}_1(u),p^{\bullet,N}_{ih}(x)\rangle)\\
        &\cdot \nabla g\left(t-s,y-\left(x+\int_{ih}^s b(r,x, \langle \tilde{K}_1(u),p^{\bullet,N}_{ih}(x)\rangle)dr + w\right)\right)\\
        &\times \left[p^{u+\kappa,N}_{ih}(x)-p^{u,N}_{ih}(x)\right]g(s-ih,w)dwdx\Bigg|
        \end{aligned}\\
        &\lesssim \frac{\|b\|_{L^\infty}}{\sqrt{t-s}}\int_{\R^d}
        \begin{aligned}[t]
        	&g\left(2(t-ih),y-\left(x+\int_{ih}^s b(r,x, \langle \tilde{K}_1(u),p^{\bullet,N}_{ih}(x)\rangle)dr\right)\right)\\
        	&\times\left|p^{u+\kappa,N}_{ih}(x)-p^{u,N}_{ih}(x)\right|dx
        \end{aligned}\\
        &\lesssim e^{(t-ih)\|b\|_{L^\infty}^2}\frac{\|b\|_{L^\infty}}{\sqrt{t-s}}\int_{\R^d} g(2(t-ih), y-x)\left|p^{u+\kappa,N}_{ih}(x)-p^{u,N}_{ih}(x)\right|dx,
    \end{align*}
    where we used Lemma~\ref{appendix.lemma.hao2021euler.propertiesOfFundamentalSolution.1} (i) and (iii) in the first inequality, and (ii) in the second inequality.\\
    
    \vspace{1em}
Collecting all the previous estimates, we arrive at
\begin{align}
	|p_t^{u+\kappa,N}(y)&-p_t^{u,N}(y)|\notag\\
       &\lesssim \int_{\R^d} g(t,x-y)|\nu_0^{u+\kappa}-\nu_0^u|(dx)
        + \int_{h}^{t} \frac{1}{\sqrt{t-s}}ds
        \|\tilde{K}_1(u+\kappa)-\tilde{K}_1(u)\|_{\mathrm{TV}}^\alpha \notag\\
        &\quad + \int_h^t\frac{h^\beta}{ (t-s)^{1/2+\beta}}ds\|\tilde{K}_1(u+\kappa)-\tilde{K}_1(u)\|_{\mathrm{TV}}^{\alpha\beta} \notag\\
        &\quad + \sum_{i=1}^{k-1} \int_{ih}^{(i+1)h}\frac{1}{\sqrt{t-s}}\int_{\R^d} g(2(t-ih), y-x)\left|p^{u+\kappa,N}_{ih}(x)-p^{u,N}_{ih}(x)\right|dxds\label{proof.claim.ArzelaAscoli.compactnessInL1:inequality.summary}
\end{align}
    
\vspace{1em}
\noindent\textbf{Claim:} For all $i \in \{1,\dots,k\}$
\begin{align*}
	&\int_{\R^d} g(2(t-ih), y-x)\left|p^{u+\kappa,N}_{ih}(x)-p^{u,N}_{ih}(x)\right|dx \\
&\lesssim \int_{\R^d} g(2t,y-z)|\nu_0^{u+\kappa}-\nu_0^u|(dz)
		+\|\tilde{K}_1(u+\kappa)-\tilde{K}_1(u)\|_{\mathrm{TV}}^\alpha\\
		&\quad +\|\tilde{K}_1(u+\kappa)-\tilde{K}_1(u)\|_{\mathrm{TV}}^{\alpha\beta}.
\end{align*}
\begin{proof}[Proof of Claim:]
	By \eqref{proof.claim.ArzelaAscoli.compactnessInL1:inequality.summary}, and Lemma~\ref{appendix.lemma.hao2021euler.propertiesOfFundamentalSolution.1} (i), and using that $g(t,x) \leq 2^d g(4t,x)$, we have
	\begin{align*}
		\int_{\R^d}& g(2(t-ih), y-x)\left|p^{u+\kappa,N}_{ih}(x)-p^{u,N}_{ih}(x)\right|dx\\
		&\lesssim \int_{\R^d} g(2t,y-z)|\nu_0^{u+\kappa}-\nu_0^u|(dz)
        + \int_{h}^{ih} \frac{1}{\sqrt{ih-s}}ds
        \|\tilde{K}_1(u+\kappa)-\tilde{K}_1(u)\|_{\mathrm{TV}}^\alpha \notag\\
        &\quad + \int_h^{ih}\frac{h^\beta}{ (ih-s)^{1/2+\beta}}ds\|\tilde{K}_1(u+\kappa)-\tilde{K}_1(u)\|_{\mathrm{TV}}^{\alpha\beta} \notag\\
        &\quad + \sum_{j=1}^{i-1} \int_{jh}^{(j+1)h}\frac{1}{\sqrt{ih-s}}\int_{\R^d} g(2(t-jh), y-z)\left|p^{u+\kappa,N}_{jh}(z)-p^{u,N}_{jh}(z)\right|dzds.
	\end{align*}
		Since
		\begin{align*}
			\int_h^{ih}\frac{1}{\sqrt{ih-s}} = \int_0^{(i-1)h}\frac1 {\sqrt s}ds = 2\sqrt{(i-1)h},
		\end{align*}
		and, choosing any $\beta\in (0,1/2)$ instead of $\beta\in (0,1)$ in the preceding calculations,
		\begin{align*}
				\int_h^{ih}\frac{h^\beta}{(ih-s)^{1/2+\beta}}ds=\frac{h^\beta}{1/2-\beta} ((i-1)h)^{1/2-\beta}\leq \frac{1}{1/2-\beta} \sqrt{(i-1)h},
		\end{align*}
	we obtain
	\begin{align*}
		&\int_{\R^d} g(2(t-ih), y-x)\left|p^{u+\kappa,N}_{ih}(x)-p^{u,N}_{ih}(x)\right|dx\\
		&\lesssim \int_{\R^d} g(2t,y-z)|\nu_0^{u+\kappa}-\nu_0^u|(dz)
		+\|\tilde{K}_1(u+\kappa)-\tilde{K}_1(u)\|_{\mathrm{TV}}^\alpha\\
		&\!\!\quad +\|\tilde{K}_1(u+\kappa)-\tilde{K}_1(u)\|_{\mathrm{TV}}^{\alpha\beta}\\
		&\!\!\quad + \sum_{j=1}^{i-1}\left(\!\sqrt{(i-j)h}-\!\sqrt{(i-(j+1))h}\right)\int_{\R^d} \!g(2(t-jh), y-\!z)\!\left|p^{u+\kappa,N}_{jh}(z)-p^{u,N}_{jh}(z)\right|\!dz\\
		&\lesssim \int_{\R^d} g(2t,y-z)|\nu_0^{u+\kappa}-\nu_0^u|(dz)
		+\|\tilde{K}_1(u+\kappa)-\tilde{K}_1(u)\|_{\mathrm{TV}}^\alpha\\
		&\!\!\quad +\|\tilde{K}_1(u+\kappa)-\tilde{K}_1(u)\|_{\mathrm{TV}}^{\alpha\beta}\\
		&\!\!\quad + \sqrt{h}\sum_{j=1}^{i-1}\frac{1}{\sqrt{(i-j)}}\int_{\R^d} g(2(t-jh), y-z)\left|p^{u+\kappa,N}_{jh}(z)-p^{u,N}_{jh}(z)\right|dz,
	\end{align*}
	where in the last inequality we used the elementary fact that $\sqrt{k}-\sqrt{k-1} \leq \frac{1}{\sqrt{k}}$ for all $k\in \N$.
Hence, by Lemma~\ref{appendix.lemma.mckee1982gronwall.gronwall}, the claim is proved.
\end{proof}

Using the previous claim and Lemma~\ref{appendix.lemma.hao2021euler.propertiesOfFundamentalSolution.1} (i), we can further estimate the right-hand side of \eqref{proof.claim.ArzelaAscoli.compactnessInL1:inequality.summary} as follows:
\begin{align*}
	&|p_t^{u+\kappa,N}(y)-p_t^{u,N}(y)|\\
	&\quad \lesssim \int_{\R^d} g(2t,x-y)|\nu_0^{u+\kappa}-\nu_0^u|(dx)\notag
        + \int_0^t \frac{1}{\sqrt{t-s}}ds
        \|\tilde{K}_1(u+\kappa)-\tilde{K}_1(u)\|_{\mathrm{TV}}^\alpha \notag\\
        &\quad + \int_0^t\frac{h^\beta}{ (t-s)^{1/2+\beta}}ds\|\tilde{K}_1(u+\kappa)-\tilde{K}_1(u)\|_{\mathrm{TV}}^{\alpha\beta} \notag\\
        &\quad + \int_0^t\frac{1}{\sqrt{t-s}}ds\Bigg(\int_{\R^d} g(2t,y-z)|\nu_0^{u+\kappa}-\nu_0^u|(dz)
		+\|\tilde{K}_1(u+\kappa)-\tilde{K}_1(u)\|_{\mathrm{TV}}^\alpha\\
		&\quad +\|\tilde{K}_1(u+\kappa)-\tilde{K}_1(u)\|_{\mathrm{TV}}^{\alpha\beta} \Bigg)\\
		&\lesssim
		\int_{\R^d}|v_0^{u+\kappa}(y)-v_0^u(y)|dy
		+\|\tilde{K}_1(u+\kappa)-\tilde{K}_1(u)\|_{\mathrm{TV}}^\alpha+ \|\tilde{K}_1(u+\kappa)-\tilde{K}_1(u)\|_{\mathrm{TV}}^{\alpha\beta}.
\end{align*}
This proves \eqref{proof.claim.ArzelaAscoli.compactnessInL1:step1} and, therefore, concludes Step 1 of this proof.

\subsection{Proof of Step 2}
    
    Let $t \in (kh,(k+1)h)$ for some $k\in \{1,\dots,N\}$, and $u \in [0,1]$. Then
    \begin{align*}
    	p_t^{u,N}(y)= \int_{\R^d} g\left(t-kh,z+\int_{kh}^t b(s,z, \langle \tilde{K}_1(u),p^{\bullet,N}_{kh}(z)\rangle)ds-y\right)p_{kh}^{u,N}(z)dz.
    \end{align*}
    Hence, by Step 1, Lemma~\ref{appendix.lemma.hao2021euler.propertiesOfFundamentalSolution.2}, Lemma~\ref{appendix.lemma.hao2021euler.propertiesOfFundamentalSolution.1} (ii), Lemma~\ref{lemma.heatkernelestimate.2} (i), \eqref{proof.claim.ArzelaAscoli.compactnessInL1:boundedDensities}, and \eqref{proof.claim.ArzelaAscoli.compactnessInL1:step1}, for every $\beta\in (0,1/2)$, there exists $C_\beta \in (0,\infty)$ such that
    \begin{align*}
    	&|p_t^{u+\kappa,N}(y)-p_t^{u,N}(y)|\\
       	&\lesssim \frac{C_\beta e^{(t-kh)\|b\|_{L^\infty}^2}}{(t-kh)^{\beta}}\!\int_{\R^d}\!\left|\int_{kh}^t b(s,z, \langle \tilde{K}_1(u+\kappa),p^{\bullet,N}_{kh}(z)\rangle)-b(s,z, \langle \tilde{K}_1(u),p^{\bullet,N}_{kh}(z)\rangle) ds\right|^\beta  \\
       	&\qquad \qquad\qquad\qquad \times g(8(t-kh),z-y)p_{kh}^{u,N}(z)dz\\
    	&\quad + \int_{\R^d} g\left(t-kh,z+\int_{kh}^t b(s,z, \langle \tilde{K}_1(u),p^{\bullet,N}_{kh}(z)\rangle)ds-y\right)\left|p_{kh}^{u+\kappa,N}(z)-p_{kh}^{u,N}(z)\right|dz\\
    	&\leq C_\beta \|v_0^\cdot\|_{L^\infty}^{\alpha\beta} e^{(t-kh)\|b\|_{L^\infty}^2}\|\tilde{K}_1(u+\kappa)-\tilde{K}_1(u)\|_{\mathrm{TV}}^{\alpha\beta}   \int_{\R^d}g(8t,z-x)d\nu_0^u(dx)dz\\
    	&\quad+ e^{(t-kh)\|b\|_{L^\infty}^2}\int_{\R^d} g\left(2(t-kh),y-z\right)\left|p_{kh}^{u+\kappa,N}(z)-p_{kh}^{u,N}(z)\right|dz\\
    	&\lesssim_{t}\int_{\R^d}|v_0^{u+\kappa}(y)-v_0^u(y)|dy
    	+\|\tilde{K}_1(u+\kappa)-\tilde{K}_1(u)\|_{\mathrm{TV}}^\alpha + \|\tilde{K}_1(u+\kappa)-\tilde{K}_1(u)\|_{\mathrm{TV}}^{\alpha\beta}.
    \end{align*}
    This concludes the proof of Step 2.

\subsection{Proof of Step 3}

Choosing $\beta>0$ small enough in Step 2, and integrating over $(0,1)$, we obtain, by Step 2, (H1), Remark~\ref{remark.hypotheses} (ii), and (H2),
\begin{align*}
    &\int_0^1\left|p^{u+\kappa,N}_t(y)-p^{u,N}_t(y)\right|du\\
   &\lesssim_{t} \int_0^1\int_{\R^d}|v_0^{u+\kappa}(y)-v_0^u(y)|dydu
   +\int_0^1\|\tilde{K}_1(u+\kappa)-\tilde{K}_1(u)\|_{\mathrm{TV}}^\alpha du\\
   &\quad + \int_0^1\|\tilde{K}_1(u+\kappa)-\tilde{K}_1(u)\|_{\mathrm{TV}}^{\alpha\beta}du
    \longrightarrow 0, \text{ as } |\kappa|\to 0.
\end{align*}
This completes the proof of Step 3 and, therefore, the proof of Claim \ref{claim.ArzelaAscoli.compactnessInL1}.\qed

\section{Proof of Theorem~\ref{theorem.GMVSDE.strongExistence}}\label{section.proof.GMVSDE.strongExistence}
Let $(p^u_t)_{t\in [0,T]}, u\in [0,1],$ be the solution to \eqref{GFPE} with $a\equiv 1_{d\times d}$, and with initial condition $(\nu_0^u)_{u\in[0,1]}$ provided by Theorem~\ref{theorem.GMVSDE.weakExistence}.
We set
\begin{align*}
	P:=\Big\{ [&0,1] \ni u\mapsto Q^u \in \mathcal{P}(\mathcal{C}_{T}^d): Q^\cdot \text{ is Lebesgue measurable}, \\ &(Q^u\circ \pi_t^{-1}) (dy)du=p^u_t(y)dydu \ \forall t\in (0,T], (Q^u\circ \pi_0^{-1}) (dy)du=\nu_0^u(dy)du\Big\}.
\end{align*}
Note that all weak solutions $(X,W)$ to \eqref{GMVSDE} with $\sL_{X^\cdot}:=[u\mapsto\sL_{X^u}(dx)] \in P$ satisfy the following system of SDEs
\begin{align}\label{theorem.strongSolution:GMVSDE.fixed.p}\begin{cases}\tag{GMVSDE$_{p}$}
	d X_t^u  &= b\Big(t,X_t^u,  \langle\tilde{K}_1(u), p^\bullet_t(X_t^u)\rangle  \Big) dt   + \sqrt{2}d W_t^u,\quad  t\in [0,T],\quad u\in [0,1].\\
	\sL_{X^\cdot}(dx) &\in P.
\end{cases}
\end{align}
For \eqref{theorem.strongSolution:GMVSDE.fixed.p}, we consider analogous solution concepts as for \eqref{GMVSDE}, see Section \ref{section.solutionConcepts.GMVSDE}.
Note that \eqref{theorem.strongSolution:GMVSDE.fixed.p} is equivalent to \eqref{GMVSDE} when prescribing the one-dimensional time marginal law densities to be $(p^u_\cdot)_{u\in [0,1]}$ in its formulation.

In the following, we construct a strong solution to \eqref{theorem.strongSolution:GMVSDE.fixed.p} by proving a restricted version of the Yamada--Watanabe theorem on Fubini extensions. Here, it is convenient to adapt the proof of \cite[Theorem 1.3.1]{grube2023thesis}, which is a modification of \cite[Chapter 9]{liu2015SPDE}.

Let $(X,W)$ denote the weak solution to \eqref{theorem.strongSolution:GMVSDE.fixed.p} with initial condition $(\nu_0^u)_{u\in [0,1]}$ provided by Theorem~\ref{theorem.GMVSDE.weakExistence}. 
According to the proof of Theorem~\ref{theorem.GMVSDE.weakExistence},  $[0,1]\ni u\mapsto\sL_{(X^u,W^u)} \in \sP(\mathcal{C}_T^d\times\mathcal{C}_{0,T}^d)$ has a Borel version (w.r.t. Lebesgue measure) which we denote by $u\mapsto\mathbb{Q}^u$. Here, $\mathbb{Q}^u$ is considered on the measurable space
	\begin{align*}
		(\Omega, \sF)
		:=(\mathcal{C}_{T}^d\times \mathcal{C}_{0,T}^d, \sB(\mathcal{C}_{T}^d)\otimes \sB(\mathcal{C}_{0,T}^d)).
	\end{align*}
	By $\Pi_0$ and $\Pi_1$ we denote the canonical projections from $\Omega$ onto its first, and second factor, respectively.
	Hence, considering the measure
	\begin{align*}
		\hat{\mathbb{Q}}^u(dx,dw_1,dw):= \mathbb{Q}^u \circ (\pi_0(\Pi_0),\Pi_0,\Pi_1)^{-1}(dx,dw_1,dw)
	\end{align*}on the measurable space
	\begin{align*}
		(\hat\Omega,\hat{\sF})
		=(\R^d \times \mathcal{C}_{T}^d\times \mathcal{C}_{0,T}^d, \sB(\R^d) \otimes \sB(\mathcal{C}_{T}^d)\otimes \sB(\mathcal{C}_{0,T}^d)),
	\end{align*}
	we conclude that the map $u \mapsto \hat{\mathbb{Q}}^u$ is Borel measurable.
	By \cite[Theorem 9.27]{kallenberg2021foundations}, we may disintegrate $\hat{\mathbb{Q}}^u$ in the form
	\begin{align*}
		\hat{\mathbb{Q}}^u(dx,dw_1,dw)=\hat{q}(x,u,w,dw_1)P^W(dw)\nu_0^u(dx),
	\end{align*}
	where $P^W$ denotes the Wiener measure on $(\mathcal{C}_{0,T}^d,\sB(\mathcal{C}_{0,T}^d))$ and 
	$$\R^d\times [0,1]\times\mathcal{C}_{0,T}^d\ni(x,u,w)\mapsto\hat{q}(x,u,w,dw_1)\in \mathcal{P}(\mathcal{C}_{T}^d)$$ is Borel measurable.
	Analogously to, e.g., \cite[Lemma E.0.10 (iii)]{liu2015SPDE},
	 one verifies that
	 for all $\sB_t(\mathcal{C}_{T}^d)$-measurable functions $f:\mathcal{C}_{T}^d\to [0,\infty)$
	 the map
	\begin{align}\label{proof.theorem.strongSolution:adaptedness.q}
        (x,u,w) \mapsto \int_{\mathcal{C}_{T}^d} f(w_1)\ \hat{q}(x,u,w,dw_1)
	\end{align}
	is $\overline{\sB(\R^d)\otimes\sB([0,1])\otimes \sB_t(\mathcal{C}_{0,T}^d)}^{\nu_0^udu\otimes P^W}$-measurable.
	
	Now consider the probability measure
		\begin{align*}
			\hat{\hat{\mathbb{Q}}}^u(dx,dw_1,dw_2,dw):= \hat{q}(x,u,w,dw_1)\hat{q}(x,u,w,dw_2)P^W(dw)\nu_0^u(dx).
		\end{align*}
		on the measurable space
		\begin{align*}
		(\R^d \times \mathcal{C}_{T}^d\times \mathcal{C}_{T}^d\times \mathcal{C}_{0,T}^d, \sB(\R^d)\otimes \sB(\mathcal{C}_{T}^d)\otimes \sB(\mathcal{C}_{T}^d)\otimes \sB(\mathcal{C}_{0,T}^d)).
		\end{align*}
		Now we consider the completion of the latter, given by the measurable space 
		$(\hat{\hat{\Omega}},\hat{\hat{\sF}}^u)$ defined as
		\begin{align*}
			\hat{\hat{\Omega}}&:=\R^d \times \mathcal{C}_{T}^d\times \mathcal{C}_{T}^d\times \mathcal{C}_{0,T}^d,\\
			\hat{\hat{\sF}}^u&:=\overline
			{\sB(\R^d)\otimes \sB(\mathcal{C}_{T}^d)\otimes \sB(\mathcal{C}_{T}^d)\otimes \sB(\mathcal{C}_{0,T}^d)}^{\hat{\hat{\mathbb{Q}}}^u},
		\end{align*}
		and consider it together with the following filtration on $(\hat{\hat{\Omega}},\hat{\hat{\sF}}^u)$:
		\begin{align*}
			\hat{\hat{\sF}}_t^u
			&:= \bigcap_{\varepsilon>0}\sigma(\sB(\R^d)\otimes \sB_{t+\varepsilon}(\mathcal{C}_{T}^d)\otimes \sB_{t+\varepsilon}(\mathcal{C}_{T}^d)\otimes \sB_{t+\varepsilon}(\mathcal{C}_{0,T}^d),\hat{\hat{\mathcal{N}}}^u),\quad t\in [0,T],\\
			\hat{\hat{\mathcal{N}}}^u&:=\{N \in \hat{\hat{\sF}}^u : \hat{\hat{\mathbb{Q}}}^u(N)=0\}.
		\end{align*}
		Furthermore, we introduce $\hat{\hat{\Pi}}_0$, $\hat{\hat{\Pi}}_1$, $\hat{\hat{\Pi}}_2$, $\hat{\hat{\Pi}}_3$ as the canonical projects from $\hat{\hat{\Omega}}$ onto its first, second, third, and fourth coordinate, respectively.
	By Theorem~\ref{appendix.fubiniextension.theorem.existence}, there exists a complete probability space $(\bar{\Omega},\bar{\sF},\bar{\P})$, an extension $([0,1];\bar{\mathcal{I}},\bar{\lambda})$ of the classical Lebesgue space on the unit interval, a Fubini extension $([0,1]\times\bar{\Omega},\bar{\mathcal{I}}\boxtimes\bar{\sF},\bar{\lambda}\boxtimes\bar{\P})$ of their canonical product space, and an $\bar{\mathcal{I}}\boxtimes\bar{\sF}$-measurable map $G: [0,1]\times\bar{\Omega} \to \hat{\hat{\Omega}}$,  such that $(G(u,\cdot))_{u\in[0,1]}$ is e.p.i., and for all $u \in [0,1]$,
	\begin{align*}
		\bar{\P}\circ G(u,\cdot)^{-1}  =\hat{\hat{\mathbb{Q}}}^u.
	\end{align*}
	On $(\bar{\Omega},\bar{\sF})$ we define the filtrations
	\begin{align*}
			\bar{\sF}_t^u
			&:= \bigcap_{\varepsilon>0}\sigma\left(G(u,\cdot)^{-1}\Big(\sB(\R^d)\otimes \sB_{t+\varepsilon}(\mathcal{C}_{T}^d)\otimes \sB_{t+\varepsilon}(\mathcal{C}_{T}^d)\otimes \sB_{t+\varepsilon}(\mathcal{C}_{0,T}^d)\Big),\bar{\mathcal{N}}\right),\\
			\bar{\mathcal{N}}&:=\{N \in \bar{\sF} : \bar{\P}(N)=0\},
		\end{align*}
	for $t\in [0,T]$.
	The maps 
	\begin{align*}
		\bar{\Pi}^u_i:=\hat{\hat{\Pi}}_i\circ G(u,\cdot),\quad u \in [0,1],\quad i=1,2,3,
	\end{align*}
	have the following properties, which can be verified with the help of easy modifications of the proofs of \cite[Lemma 1.3.4, Lemma 1.3.5]{grube2023thesis} (cf. \cite[Lemma E.0.11, Lemma E.0.12]{liu2015SPDE}):
	\begin{enumerate}
		\item For $du$-a.e. $u \in [0,1]$, $\bar{\Pi}_3^u$ is an $(\bar{\sF}^u_t)$-Wiener process,
        and $(\bar{\Pi}_2^u)_{u\in[0,1]}$ is e.p.i.;
		\item $(\bar{\Pi}_i,\bar{\Pi}_3)$ is a weak solution to \eqref{theorem.strongSolution:GMVSDE.fixed.p} with $\bar{\Pi}_i(0)=\bar{\Pi}_0$ $\bar{\lambda}\boxtimes\bar{\P}$-a.s., and $\sL_{\bar{\Pi}_i^u}=\mathcal L _{X^u}$, for $du$-a.e. $u \in [0,1]$, and $i=1,2$.
	\end{enumerate}
	Note that pathwise uniqueness holds for \eqref{theorem.strongSolution:GMVSDE.fixed.p}, which can be easily deduced from \cite[Theorem 1]{veretennikov1980strong}.
	This allows us to proceed similarly to the proof of \cite[Lemma 1.3.7]{grube2023thesis} (cf. \cite[Lemma E.0.12]{liu2015SPDE}) and conclude that
	\begin{align*}
		1&= (\bar{\lambda}\boxtimes\bar{\P})(\{\bar{\Pi}_1=\bar{\Pi}_2\}) \\
		&= \int_0^1\int_{\R^d}\int_{\mathcal{C}_{T}^d}\int_{\mathcal{C}_{T}^d}\int_{\mathcal{C}_{0,T}^d} 1_D(w_1,w_2)(x,u,w_1,w_2,w)\hat{q}(u,x,w,dw_1)\hat{q}(u,x,w,dw_2)\\
        &\quad \quad \quad \quad \quad \quad \quad \quad \quad \quad \quad \times P^W(dw)\nu_0^u(dx)du,
	\end{align*}
	where
	\begin{align*}
		D:= \{(w_1,w_2) \in \mathcal{C}_T^d\times\mathcal{C}_T^d: w_1=w_2\} \in \sB(\mathcal{C}_{T}^d)\otimes \sB(\mathcal{C}_{T}^d),
	\end{align*}
	so that, by \cite[Lemma E.0.9]{liu2015SPDE}, for $\nu_0^u(dx)du\otimes P^W$-a.e. $(x,u,w)\in \R^d\times [0,1]\times \mathcal{C}_{0,T}^d$
	\begin{align*}
		\hat{q}(x,u,w,dw_1)=\delta_{F(x,u,w)}(dw_1)\quad \text{ on } (\mathcal{C}_{T}^d,\sB(\mathcal{C}_{T}^d)),
	\end{align*}
	for some function $F:\R^d\times [0,1]\times \mathcal{C}_{0,T}^d\to \mathcal{C}_{T}^d$.
	Setting $F$ equal to zero on the negligible set, for which the last equality does not hold, we obtain a 
	$\overline{\sB(\R^d)\otimes \sB([0,1])\otimes \sB(\mathcal{C}_{0,T}^d)}^{\nu_0^udu \otimes P^W}$-measurable version, which we consider from now on.
    Similar to \cite[Proof of Lemma 1.3.7]{grube2023thesis}, one shows that the adaptedness property of $\hat{q}$, as expressed in \eqref{proof.theorem.strongSolution:adaptedness.q}, implies the following adaptedness property for $F$:
	For $\nu_0^u du$-a.e. $(x,u)\in \R^d\times [0,1]$,
	\begin{align*}
		\mathcal{C}_{0,T}^d \ni w\mapsto F(u,x,w)\in \mathcal{C}_T^d
	\end{align*}
	is $\overline{\sB_t(\mathcal{C}_{0,T}^d)}^{P^W}/\sB_t(\mathcal{C}_{T}^d)$-measurable.
	
	Moreover, similar to \cite[Lemma 1.3.8]{grube2023thesis} (cf. \cite[Lemma E.0.14]{liu2015SPDE}), one verifies that, for $du$-a.e. $u \in [0,1]$,
	\begin{align*}
		X^u(\omega)=F(X^u(\omega)(0),u,W^u(\omega)),\quad \text{ for }\P\text{-a.e. }  \omega \in \Omega.
	\end{align*}
	
	Also, analogously to \cite[Lemma 1.3.9]{grube2023thesis} (cf. \cite[Lemma E.0.15]{liu2015SPDE}), one proves the following:
    Let $([0,1]\times\tilde{\Omega},\tilde{\mathcal{I}}\boxtimes\sF,\tilde{\lambda}\boxtimes\tilde{\P})$ be a Fubini extension and let $(\tilde{\sF}_\cdot^u)_{u\in [0,1]}$ be a family of filtrations on $(\tilde{\Omega},\tilde{\sF})$ as in Definition \ref{definition.GMVSDE.weakSolution} (i), (ii) such that the probability space supports a map $(u,\omega)\mapsto (\tilde{X}_0^u(\omega),\tilde{W}^u(\omega))$ fulfilling the conditions in Definition \ref{definition.GMVSDE.weakSolution} (iii)-(vii), (ix).
    Then
	\begin{align*}
		(\tilde{X},\tilde{W})=(X^u,\tilde{W}^u)_{u\in[0,1]},
	\end{align*}
	where $\tilde{X}^u:=F(\tilde{X}_0^u,u,\tilde{W}^u)$, is a weak solution to \eqref{GMVSDE} with $\tilde{X}(0)=\tilde{X}_0$ $\tilde{\lambda}\boxtimes\tilde{\P}$-a.s.
	
	This completes the proof.\qed
	
\section{Proof of Theorem~\ref{theorem.GMVSDE.strongExistence.uniform}}\label{section.proof.GMVSDE.strongExistence.uniform}
Let $P^W$ denote the Wiener measure on $(\mathcal{C}_{0,T}^d,\sB(\mathcal{C}_{0,T}^d))$.
Consider the probability measure
\begin{align*}
	Q(dx, du,dw_1,dw) := 
	\delta_{F(x,u,w)}(dw_1) P^W(dw) \nu_0^u(dx)du,
\end{align*}
on the measurable space
\begin{align*}
	(\R^d \times [0,1] \times \mathcal{C}_{T}^d \times \mathcal{C}_{0,T}^d, \sB(\R^d) \otimes \sB([0,1]) \otimes \sB(\mathcal{C}_{T}^d) \otimes \sB(\mathcal{C}_{0,T}^d)).
\end{align*}
We define $(\Omega,\sF,Q)$ to be the probability space obtained by completing with respect to $Q$, i.e.,
\begin{align*}
	\Omega:=\R^d \times [0,1] \times \mathcal{C}_{T}^d \times \mathcal{C}_{0,T}^d,\quad
	\sF:=\overline{\sB(\R^d) \otimes \sB([0,1]) \otimes \sB(\mathcal{C}_{T}^d) \otimes \sB(\mathcal{C}_{0,T}^d)}^{Q},
\end{align*}
and consider it together with the filtration
\begin{align*}
	\sF_t
	&:= \bigcap_{\varepsilon>0}(\sB(\R^d)\otimes\sB([0,1])\otimes \sB_{t+\varepsilon}(\mathcal{C}_{T}^d)\otimes \sB_{t+\varepsilon}(\mathcal{C}_{0,T}^d),\mathcal{N}),\quad t\in [0,T],\\
			\mathcal{N}&:=\{N \in \sF : Q(N)=0\}.
\end{align*}
In the following, $\Pi_0, \Pi_1,\Pi_2,\Pi_3$ denote the canonical projections from $\Omega$ onto its first, second, third, and fourth coordinate, respectively.

In the following, we follow the idea of proof in \cite[Lemma 1.3.9.]{grube2023thesis}.
First, by the independence of $(U,X_0)$ and $W$, and by the proof of Theorem~\ref{theorem.GMVSDE.strongExistence}, we obtain
\begin{align*}
	\P(\{X_0=F(X_0,U,W)(0)\})
	&=\int_0^1 \int_{\R^d}\int_{C_{0,T}^d} 1_{\{x=F(x,u,w)\}}P^W(dw) \nu_0^u(dx)du\\
	&= \int_0^1\bar{\P}(\bar{\Pi}_0^u=F(\bar{\Pi}_0^u,u,\bar{\Pi}_2^u))du
	=1,
\end{align*}
where we use the notation from the proof of Theorem~\ref{theorem.GMVSDE.strongExistence} in the second inequality.
We define the $\sF$-measurable set
\begin{align*}
	A:=\left\{\Pi_2(t)-\Pi_0 =\int_0^t b(s,\Pi_2(s), \langle \tilde{K}_1(\Pi_1),p_s^\bullet (\Pi_2(s))\rangle)ds + \sqrt{2}d\Pi_3(t)\ \ \forall t\in [0,T]\right\}.
\end{align*}
Also, by the proof of Theorem~\ref{theorem.GMVSDE.strongExistence} and, again, using the notation therein, we have
\begin{align*}
	\P&(\{(X_0,U,F(X_0,U,W),W) \in A\})\\
    &\quad =Q(A)=\int_0^1 \bar{\P}\left(\left\{(\bar{\Pi}_0^u,F(\bar{\Pi}_0^u,u,\bar{\Pi}_2^u),\bar{\Pi}_2^u)\in \hat{A}^u\right\}\right)du=1,
\end{align*}
where $\hat{A}^u \in \sB(\R^d)\otimes\sB (C_T^d)\otimes \sB (C_{0,T}^d)$ is defined as
\begin{align*}
	\hat{A}^u:=\left\{\hat{\Pi}_2(t)-\hat{\Pi}_0 =\int_0^t b(s,\hat{\Pi}_2(s), \langle \tilde{K}_1(u),p_s^\bullet (\hat{\Pi}_2(s))\rangle)ds + \sqrt{2}d\hat{\Pi}_3(t)\ \forall t\in [0,T]\right\}.
\end{align*}
Here, $\hat{\Pi}_i$ denotes the canonical projections from $\hat{\Omega}$ onto its $i$-th coordinate, for $i \in \{0,1,2\}$.

This completes the proof. \qed

\section{Proof of Corollary~\ref{corollary.GMVSDE.uniqueStrongSolution}}\label{section.proof.GMVSDE.uniqueStrongSolution}
Let $(X,W), (\bar{X},W)$ be two weak solutions to \eqref{GMVSDE} on the same probability space, driven by the same e.p.i. Wiener processes $W^u$, $u\in [0,1]$, and satisfying $\bar{X}(0)=X(0)$ $\lambda\boxtimes\P$-a.s.
Using a similar argument as in \cite[(3.13), (3.14)]{hao2021euler}, based on Gaussian upper bounds for the density of the solution to an ordinary SDE with additive noise and uniformly bounded drift coefficient, one obtains, due to the assumption $v_0 \in L^\infty(\R^d\times [0,1])$, that
\begin{align*}
	[(u,t,x)\mapsto p^u_t(x)],[(u,t,x)\mapsto \bar{p}^u_t(x)] \in L^\infty([0,1]\times[0,T]\times\R^d),
\end{align*}
where $\sL_{X^u(t)}(dx)du= p^u_t(x)dxdu$ and $\sL_{\bar{X}^u(t)}(dx)du= \bar{p}^u_t(x)dxdu$, for all $t\in (0,T]$.
By Theorem~\ref{theorem.GFPE.uniqueness}, we obtain that
\begin{align}
	\sL_{X^u(t)}(dx)du=\sL_{\bar{X}^u(t)}(dx)du=p^u_t(x)dxdu\quad \forall t\in (0,T],
\end{align}
where $(p^u_t)_{t\in [0,T]}, u\in [0,1],$ is the solution to \eqref{GFPE} provided by Lemma~\ref{lemma.ArzelaAscoli}.
Consequently, both $(X,W)$ and $(\bar{X},W)$ are weak solutions to \eqref{theorem.strongSolution:GMVSDE.fixed.p}, where \eqref{theorem.strongSolution:GMVSDE.fixed.p} was introduced in the proof of Theorem~\ref{theorem.GMVSDE.strongExistence} (see Section \ref{section.proof.GMVSDE.strongExistence}).
As already observed there, pathwise uniqueness for \eqref{theorem.strongSolution:GMVSDE.fixed.p} follows from \cite[Theorem 1]{veretennikov1980strong}. 
Hence, Theorem~\ref{theorem.GMVSDE.strongExistence} implies that $X=\bar{X}$ $\lambda\boxtimes\P$-a.s. This ends the proof.\qed

\section{Proof of Theorem~\ref{theorem.GFPE.uniqueness}}\label{section.proof.GFPE.uniqueness}
This proof uses a variant of the technique developed in the proof of \cite[Theorem 2.1]{barbu2021weakUniqueness}, and consequently several parts of the framework and reasoning remain unchanged.
 The main difference to \cite[Theorem 2.1]{barbu2021weakUniqueness} is that we need to deal with a continuum of solutions, indexed by $u$, to respective nonlinear Fokker--Planck equations. In order to close a comparable Gronwall argument as in the proof of \cite{barbu2021weakUniqueness}, from which we conclude the uniqueness of two solutions, we use an additional integration argument in the $u$-variable in combination with Condition (H3).

Let $(p^u_t)_{t\in [0,T]},(p^u_t)_{t\in [0,T]}$, $u \in [0,1]$, be solutions to \eqref{GFPE} with initial condition $(\nu_0^u)_{u\in [0,1]}$.
For $t\in [0,T]$, $u \in [0,1]$, we define
\begin{align*}
  \mathfrak{p}_t^u&:=p_t^u-\bar{p}^u_t,\\
    \mathfrak{b}_t^u&:=b(t,\cdot, \langle\tilde{K}_1(u),p_t^{\bullet}(\cdot)\rangle)p_t^u(\cdot)-b(t,\cdot, \langle\tilde{K}_1(u),\bar{p}_t^{\bullet}(\cdot)\rangle)\bar{p}_t^u(\cdot)\\
    &\ =  \left(b(t,\cdot, \langle\tilde{K}_1(u),p_t^{\bullet}(\cdot)\rangle)-b(t,\cdot, \langle\tilde{K}_1(u),\bar{p}_t^{\bullet}(\cdot)\rangle)\right)p_t^u(\cdot) + b(t,\cdot, \langle\tilde{K}_1(u),\bar{p}_t^{\bullet}(\cdot)\rangle)\mathfrak{p}_t^u(\cdot).
\end{align*}
Clearly, for $du$-a.e. $u \in [0,1]$, $\mathfrak{p}^u$ solves
\begin{equation}
\left\{
\begin{alignedat}{2}
\partial_t \mathfrak{p}_t^u + \operatorname{div}(\mathfrak{b}_t^u)
- \Delta\bigl(\beta(p_t^u)-\beta(\bar p_t^u)\bigr)
&= 0,
&\qquad \text{in } \mathcal{D}'((0,T)\times\mathbb{R}^d),\\
\mathfrak{p}^u|_{t=0}
&= 0.
\end{alignedat}
\right.
\end{equation}

Consider the operator $R:H^{-1}(\R^d)\to H^1(\R^d)$ defined as $R\varphi:=(\text{id}-\Delta)^{-1}\varphi$. Note that $R$ is an isomorphism and, additionally, $R\in L(H^{-2}(\R^d),L^2(\R^d))$\footnote{Here, $L(H^{-2}(\R^d),L^2(\R^d))$ denotes the set of bounded operators from $H^{-2}(\R^d)$ to $L^2(\R^d)$.}.

Since $p^u, \beta(p^u), b(\cdot,\cdot, \langle\tilde{K}_1(u),p_{\cdot}^{\bullet}(\cdot)\rangle)p_{\cdot}^u(\cdot) \in L^2([0,T];L^2(\R^d))$, and similarly for $\bar{p}^u$ replacing $p^u$, we have $\mathfrak{p}^u \in W^{1,2}([0,T];H^{-2}(\R^d))$.\\
Then $y^u:=R\mathfrak{p}^u\in L^2([0,T];H^2(\R^d))\cap W^{1,2}([0,T];L^2(\R^d))$ and $y^u$ solves
\begin{equation}\label{proof.GFPE.uniqueness.eq:H-1}
\left\{
\begin{alignedat}{2}
\frac{d}{dt} y_t^u
- R\Delta\bigl(\beta(p_t^u)-\beta(\bar p_t^u)\bigr)
+ R\operatorname{div}(\mathfrak{b}_t^u)
&= 0,
&\qquad \text{for a.e. } t\in[0,T],\\
y^u|_{t=0}
&= 0,
\end{alignedat}
\right.
\end{equation}
where $\frac{d}{dt}y^u \in L^2([0,T];L^2(\R^d))$. The time derivative is taken in the sense of $L^2(\R^d)$-valued Schwartz distributions on $(0,T)$ and so $y^u:(0,T)\to L^2(\R^d)$ is absolutely continuous.

Without loss of generality, we assume that
\begin{align}
    \mathfrak{p}^u\in L^2([0,T];H^1(\R^d))\cap W^{1,2}([0,T];H^{-1}(\R^d)).
\end{align}
This additional assumption can, in fact, be dropped by the very same argument employed in the proof of \cite[Theorem 2.1]{barbu2021weakUniqueness}, where  we use (for the first and last time) the assumption $d\geq 3$.

Taking the inner product in $L^2(\R^d)$ of \eqref{proof.GFPE.uniqueness.eq:H-1} with $\mathfrak{p}^u_t$, we obtain, for a.e. $t\in[0,T]$,
\begin{align}
    0=\left\langle \frac{d}{dt} \mathfrak{p}^u_t, \mathfrak{p}^u_t\right\rangle_{-1} -\langle R\Delta (\beta(p_t^u)-\beta(\bar{p}_t^u)), \mathfrak{p}^u_t\rangle_2 + \langle R \text{div}(\mathfrak{b}^u_t),\mathfrak{p}^u_t\rangle_2,
\end{align}
or, equivalently, since $\langle \frac d{dt} \mathfrak{p}^u_t,\mathfrak{p}_t^u\rangle_{-1} = \frac1 2\frac d {dt} \|\mathfrak{p}_t^u\|_{-1}^2$,
\begin{align}\label{proof.GFPE.uniqueness.eq:H-1:duality}
\frac{1}{2}\frac{d}{dt}\|\mathfrak{p}^u_t\|_{-1}^2 + \langle \beta(p_t^u)-\beta(\bar{p}_t^u),\mathfrak{p}^u_t\rangle_2
=\langle \beta(p_t^u)-\beta(\bar{p}_t^u),\mathfrak{p}^u_t\rangle_{-1}-\langle \text{div}(\mathfrak{b}^u_t),\mathfrak{p}^u_t\rangle_{-1}.
\end{align}
Furthermore, by Young's inequality, for every $\delta>0$ there is $C_\delta>0$ such that
\begin{align}\label{GPFE.theorem.uniqueness.estimate:1}
    |\langle \text{div}(\mathfrak{b}^u_t),\mathfrak{p}^u_t\rangle_{-1}|
    \lesssim \|\mathfrak{b}^u_t\|_2\|\mathfrak{p}^u_t\|_{-1}\lesssim \delta \|\mathfrak{b}^u_t\|_2^2 + C_{\delta} \|\mathfrak{p}^u_t\|_{-1}^2.
\end{align}
By \eqref{theorem.GFPE.uniqueness:1}, we have
\begin{align}\label{GPFE.theorem.uniqueness.estimate:2}
    \|\mathfrak{b}^u_t\|_2^2 \lesssim \|p^u_t\|_{\infty}^2\left\| \langle\tilde{K}_1(u),\mathfrak{p}_t^\bullet(\cdot)\rangle \right\|_2^2+ \|b\|_{\infty}^2\left\|
    \mathfrak{p}_t^u\right\|^2_2.
\end{align}
Here, by Fubini's theorem and (H3), we have
\begin{align*}
    \int_0^1\left\|\langle\tilde{K}_1(u),\mathfrak{p}^\cdot_t(\cdot)\rangle\right\|_2^2 du
    & = \!\int_{\R^d}\int_0^1\left|\langle\tilde{K}_1(u),\mathfrak{p}^\cdot_t(x)\rangle\right|^2dudx\\
    &\lesssim \int_{\R^d}\int_0^1\left|\mathfrak{p}^u_t(x)\right|^2dudx
    = \int_0^1\left\|\mathfrak{p}^u_s(\cdot)\right\|_2^2du.
\end{align*}
Also, by Young's inequality and (H4), for every $\eta>0$ there exists $C_\eta>0$ such that
\begin{align}\label{GPFE.theorem.uniqueness.estimate:3}
    |\langle \beta(p_t^u)-\beta(\bar{p}_t^u),\mathfrak{p}^u_t\rangle_{-1}|
    \leq \eta\sup_{|r|\leq M}\beta'(r)\ \|\mathfrak{p}_t^u\|_{2}^2 +C_\eta\|\mathfrak{p}_t^u\|_{-1}^2,
\end{align}
where $M:=\max\{\mathrm{esssup}_{t,x,u}{|p_t^u(x)|}, \mathrm{esssup}_{t,x,u}{|\bar{p}_t^u(x)|}\} <\infty$.

Now, we integrate \eqref{proof.GFPE.uniqueness.eq:H-1:duality} on both sides over $(0,T)$, which leads to
\begin{align}\label{proof.GFPE.uniqueness.eq:H-1:duality_integralForm}
    \frac12 &\|\mathfrak{p}^u_t\|_{-1}^2 
    + \int_0^t\langle \beta(p_s^u)-\beta(\bar{p}_s^u),\mathfrak{p}^u_s\rangle_2 ds \notag\\
&=\int_0^t\langle \beta(p_s^u)-\beta(\bar{p}_s^u),\mathfrak{p}^u_s\rangle_{-1} ds
    -\int_0^t\langle R \text{div}(\mathfrak{b}^u_s),\mathfrak{p}^u_s\rangle_2 ds.
\end{align}
Then, we integrate \eqref{proof.GFPE.uniqueness.eq:H-1:duality_integralForm} over $(0,1)$, and use \eqref{GPFE.theorem.uniqueness.estimate:1}, \eqref{GPFE.theorem.uniqueness.estimate:2}, \eqref{GPFE.theorem.uniqueness.estimate:3}, (H4), to obtain that, by Fubini's theorem,
\begin{align}\label{GPFE.theorem.uniqueness.estimate:4}
    &\frac12 \int_0^1\|\mathfrak{p}_t^u\|_{-1}^2du + \gamma_0\int_0^t\int_0^1\|\mathfrak{p}_s^u\|_2^2duds \notag\\
    &\lesssim_M (\delta+\eta) \int_0^t\int_0^1\|\mathfrak{p}_s^u\|_{2}^2duds
    + (C_\delta+C_\eta) \int_0^t\int_0^1\|\mathfrak{p}_s^u\|_{-1}^2duds.
\end{align}
Choosing both $\delta$ and $\eta$ small enough in \eqref{GPFE.theorem.uniqueness.estimate:4}, we obtain, for some constant $\tilde{C}=\tilde{C}(\delta,\eta)>0$, 
\begin{align}
    \int_0^1\|\mathfrak{p}_t^u\|_{-1}^2du \leq \tilde{C} \int_0^t\int_0^1\|\mathfrak{p}_s^u\|_{-1}^2duds.
\end{align}
Hence, by Gronwall's lemma, for $dt$-a.e. $t\in[0,T]$
\begin{align}
    \int_0^1\|\mathfrak{p}_t^u\|_{-1}^2du=0.
\end{align}
Since both $p,\bar{p}$ satisfy Definition \ref{definition.GFPE.solution} (ii), we conclude  from the last identity that 
$$p^u_t(x)dxdu=\bar{p}^u_t(x)dxdu\quad \forall t \in[0,T].$$
This ends the uniqueness part of the proof.

The last part of the assertion of this lemma follows directly from Theorem~\ref{theorem.GMVSDE.weakExistence}.
\qed
\section{Proof of Theorem~\ref{theorem.GFPE.stepKernel}}\label{section.proof.GFPE.stepKernel}
Let $j\in J$, $t\in (0,T]$, and $y \in \R^d$.
Furthermore, let $\{N_k\}_{k\in\N} \subseteq \N$ denote the subsequence provided by Lemma \ref{lemma.ArzelaAscoli}. We have
\begin{align*}
	\frac 1 {|A_j|}\int_{A_j}\int_{A_j} |p_t^u(y)-p_t^v(y)|dvdu
	&\leq 2\int_{A_j} |p_t^{u,N_k}(y)-p_t^u(y)|du\\
	&\quad + \frac 1 {|A_j|}\int_{A_j}\int_{A_j} |p_t^{v,N_k}(y)-p_t^{u,N_k}(y)|dvdu,
\end{align*}
where $|A_j|$ denotes the Lebesgue measure of the Borel set $A_j$.
By Lemma~\ref{lemma.ArzelaAscoli}, the first summand on the right-hand side converges to zero, as $N\to \infty$.
By Step 2 in the proof of Claim \ref{claim.ArzelaAscoli.compactnessInL1} (see Section \ref{section.proof.claim.ArzelaAscoli.compactnessInL1}), the second summand is equal to zero.
Now, setting $\bar{p}^j_t(y):= |A_j|^{-1}\int_{A_j} p_t^u(y)du\ (\in [0,\infty))$, for all $t\in (0,T], y\in \R^d$, we estimate
\begin{align*}
	0=\frac 1 {|A_j|}\int_{A_j}\int_{A_j} |p_t^u(y)-p_t^v(y)|dvdu
	\geq \int_{A_j}\left| p_t^u(y)-\bar{p}_t^j(y)\right|du.
\end{align*}
Hence, $(0,T]\times\R^d \ni (t,y)\mapsto p_t^\cdot(y)_{|A_j} \in L^1(A_j)$ takes values in the classes of $du$-a.e. constant functions on $A_j$. This ends the proof.
\qed

\appendix
\section*{Appendix}
\section{Heat kernels and their estimates}\label{appendix.heatKernel}
We introduce the function
\begin{align}
       g(t,x):= (4\pi t)^{-d/2}e^{-|x|^2/(4t)},\quad t>0,x\in\R^d,
\end{align}
which denotes the fundamental solution to the heat equation
\begin{equation}
\left\{
\begin{alignedat}{2}
\partial_t g(t,x)
&= \Delta g(t,x),
&\qquad (t,x)\in (0,\infty)\times\mathbb{R}^d,\\
g|_{t=0}
&= \delta_x.
\end{alignedat}
\right.
\end{equation}
The following lemmas summarize elementary properties of $g$ from \cite{hao2021euler}.
\begin{lemma}[{\cite[(2.2), (2.3)]{hao2021euler}}]\label{appendix.lemma.hao2021euler.propertiesOfFundamentalSolution.1}
    For all $s,t\in (0,\infty)$ and $x\in \R^d$, we have
    \begin{enumerate}[label=(\roman*)]
        \item $\int_{\R^d}g(t,x-y)g(s,y)dy = g(t+s,x)$ ('Chapman--Kolmogorov equation'),
        \item $g(t,x+y) \leq 2^{\frac d 2}g(2t,x)e^{\frac{|y|^2}{4t}}$,
        \item $|\nabla g(t,x)|\leq \frac{2^{\frac d 2}}{\sqrt{t}}g(2t,x)$.
    \end{enumerate}
\end{lemma}
\begin{lemma}[{\cite[Lemma 2.1]{hao2021euler}}]\label{appendix.lemma.hao2021euler.propertiesOfFundamentalSolution.2}
For every $T\in (0,\infty),\beta \in (0,1)$ and $j=0,1$, there is a constant $C=C(T,\beta,j,d)>0$ such that for every $t\in (0,T]$ and $x_1,x_2\in \R^d$,
\begin{align}
    |\nabla^jg(t,x_1)-\nabla^jg(t,x_2)| \leq C |x_1-x_2|^\beta t^{-\frac j 2 -\beta} \sum_{i=1}^2 g(4t,x_i),
\end{align}
and for any $0<t_1<t_2\leq T$ and $x\in \R^d$
\begin{align}
    |\nabla^jg(t_1,x)-\nabla^jg(t_2,x)| \leq C |t_1-t_2|^{\frac \beta 2}  \sum_{i=1}^2 t_i^{-\frac {j+\beta} 2} g(2t_i,x).
\end{align}
\end{lemma}
\section{A discrete Gronwall lemma}\label{appendix.gronwall}
\begin{lemma}[{\cite[Lemma 3]{mckee1982gronwall}}]\label{appendix.lemma.mckee1982gronwall.gronwall}
    Let $x_j$, $j=0,1,\dots,N$, be real numbers with
    \begin{align}
        |x_0|\leq \delta,\quad |x_i|\leq h^{1/2}M\sum_{j=0}^{i-1}\frac1 {(i-j)^{1/2}}|x_j|+\delta, \ \ i=1,2,\dots,N,
    \end{align}
    where $M>0$ and independent of $h,\delta>0$ and $T=Nh$, then
    \begin{align}
        \|x_i\|_\infty \leq \delta(1+h^{1/2}M+hM^2\pi + 2MT^{1/2})e^{M^2\pi T}, \ \ i=1,2,\dots,N.
    \end{align}
\end{lemma}

\section{On Fubini extensions}\label{appendix.fubiniextension}

The following definitions are taken (up to non-essential modifications) from \cite{sun2009risk}, which is heavily based on \cite{sun2006exact}. Furthermore, Theorem \ref{appendix.fubiniextension.theorem.existence} below is taken from \cite[Theorem 1]{aurell2022stochastic}, which is a refinement of \cite[Theorem 1]{sun2009risk}.

Due to the assumptions in \cite{sun2006exact}, we also assume in this section that all appearing probability spaces are assumed to be complete, including product probability spaces; here the completions of the sigma algebras with the zero sets are not denoted differently.

Let $(I,\mathcal{I},\lambda)$ be an atomless probability space, and $(\Omega,\sF,\P)$ be a probability space. In the following, for a map $f$ defined on $I\times\Omega$, we introduce the abbreviations $f_i:=f(i,\cdot)$ and $f_\omega:= f(\cdot,\omega)$.

\begin{definition}[{\cite[Definiton 1]{sun2009risk}}]\label{appendix.fubiniextension.definition.epi}
A process $f$ from $I \times \Omega$ to a complete separable metric space $X$ is said to be \emph{essentially pairwise independent} if, for $\lambda$-a.e. $s \in I$, the random variables $f_s$ and $f_i$ are independent for $\lambda$-a.e. $i \in I$.
\end{definition}

\begin{definition}[{\cite[Definiton 3]{sun2009risk}}]\label{appendix.fubiniextension.definition.fubiniextension}
Let $(I \times \Omega, \mathcal{I} \otimes \sF, \lambda \otimes \P)$ be the usual product probability space of the probability spaces $(I,\mathcal{I},\lambda)$ and $(\Omega,\sF,\P)$. A probability space $(I \times \Omega, \mathcal{W}, Q)$ extending $(I \times \Omega, \mathcal{I} \otimes \sF, \lambda \otimes \P)$ is said to be a \emph{Fubini extension} if, for any real-valued $Q$-integrable function $f$ on $(I \times \Omega, \mathcal{W})$,
\begin{enumerate}[label=(\roman*)]
\item the two functions $f_i$ and $f_\omega$ are integrable respectively on $(\Omega,\sF,\P)$ for $\lambda$-a.e. $i \in I$, and on $(I,\mathcal{I},\lambda)$ for $\P$-a.e. $\omega \in \Omega$;

\item $i \mapsto \int_{\Omega} f_i\,d\P$ and $\omega \mapsto \int_{I} f_\omega\,d\lambda$ are integrable respectively on $(I,\mathcal{I},\lambda)$ and $(\Omega,\sF,\P)$;

\item
\begin{align*}
\int_{I \times \Omega} f\,dQ = \int_I \left(\int_\Omega f_i\,d\P\right)d\lambda
= \int_\Omega \left(\int_I f_\omega\,d\lambda\right)d\P.
\end{align*}
\end{enumerate}

To reflect the fact that the probability space $(I \times \Omega,\mathcal{W},Q)$ has $(I,\mathcal{I},\lambda)$ and $(\Omega,\sF,\P)$ as its marginal spaces, as required by the Fubini property, it will be denoted by $(I \times \Omega,\mathcal{I} \boxtimes \sF,\lambda \boxtimes \P)$.
\end{definition}

\begin{theorem}[{cf. \cite[Theorem 1]{aurell2022stochastic}, \cite[Theorem 1]{sun2009risk}}]\label{appendix.fubiniextension.theorem.existence}
Let $I=[0,1]$ and let $X$ be a complete separable metric space. There exist a probability space $(I,\mathcal{I},\lambda)$ extending the standard Lebesgue space $(I,\mathcal{A},du)$ ($\mathcal{A}$ denoting the $\sigma$-algebra of Lebesgue measurable sets), a probability space $(\Omega,\sF,\P)$, and a Fubini extension $(I \times \Omega,\mathcal I \boxtimes \sF,\lambda\boxtimes \P)$ such that for any measurable mapping $\varphi:(I,\mathcal{I},\lambda)\to\mathcal{P}(X)$, there exists a $\mathcal I \boxtimes \sF$-measurable process $f:I\times\Omega\to X$ such that the random variables $f_i$, $i \in I$, are essentially pairwise independent and
\begin{align}\label{appendix.fubiniextension.theorem.existence:1}
\P\circ f_i^{-1}=\varphi(u),\quad \text{for all $u\in I$.}
\end{align}

\end{theorem}

The following figure visualizes the statement of Theorem~\ref{appendix.fubiniextension.theorem.existence}. See also \cite{amini2025brownianmotionfubiniextension} for a comparable illustration.

\vspace{3em}

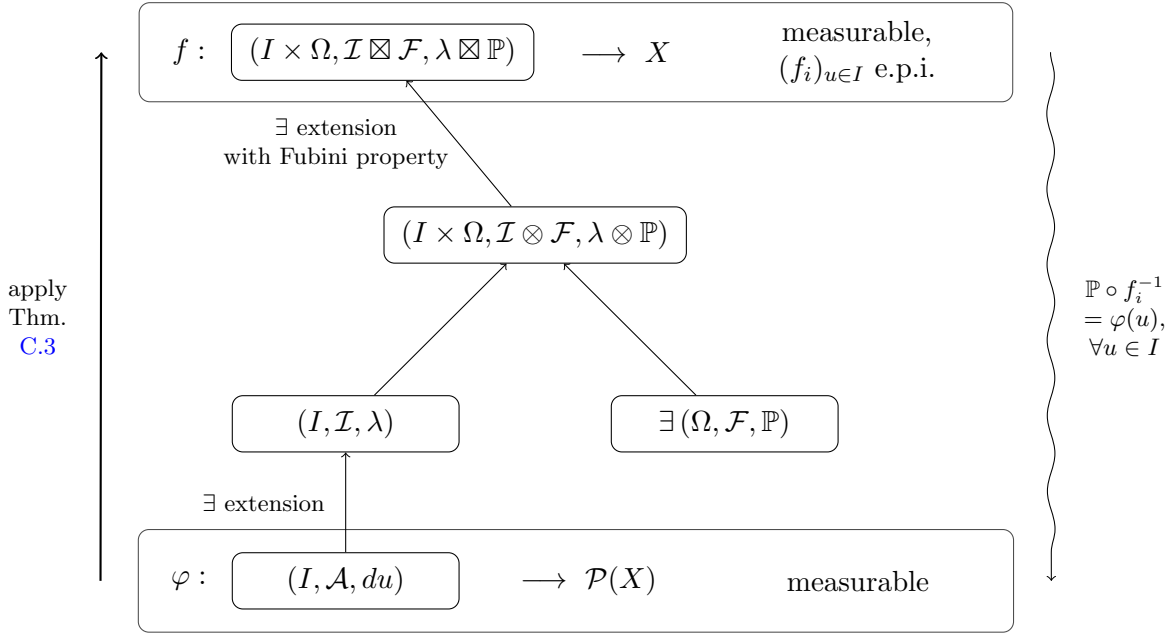
\begin{figure}[ht]
    \centering
\begin{tikzpicture}[
    probabilityspace/.style={
        draw=black,
        rounded corners,
        inner sep=1.8mm,
    },
    varphi-functionBox/.style={
        draw=black!70,
        rounded corners,
        minimum width=0.8cm,
        inner sep=3mm,
        outer sep=5mm,
    },
    f-functionBox/.style={
        draw=black!70,
        rounded corners,
        minimum width=0.8cm,
        inner sep=3mm,
        inner ysep=1mm,
        outer sep=5mm,
    },
    arrow_text/.style={
    	font=\footnotesize
    }
]

\node (varphi1)
at (0,0) {$\varphi:$};

\node[probabilityspace, right=2mm of varphi1, minimum width=3cm] (varphi2)
{$(I,\sA,du)$};

\node[right=2mm of varphi2, minimum width=3cm] (varphi3)
{$\longrightarrow\ \sP(X)$};

\node[right=3mm of varphi3, minimum width=3.5cm] (varphi4)
{measurable};

\node[varphi-functionBox, fit=(varphi1) (varphi2) (varphi3) (varphi4)] (varphi)
{};

\node[probabilityspace, minimum width=3cm](varphi2extension)
at ($(varphi2.north)+(0,1.7)$) {$(I,\mathcal{I},\lambda)$};

\node[probabilityspace, right=2cm of varphi2extension, minimum width=3cm] (omega)
{$\exists\,(\Omega,\sF,\P)$};

\node[probabilityspace, minimum width=4cm] (f2extension-omega-product)
at ($(varphi2extension)!0.5!(omega)+(0,2.5)$) {$(I\times\Omega,\mathcal{I}\otimes\sF,\lambda\otimes\P)$};

\node (f1) 
at (0,7) {$f:$};

\node[probabilityspace, right=2mm of f1, minimum width=4cm] (f2)
{$(I\times\Omega,\mathcal{I}\boxtimes\mathcal F,\lambda\boxtimes\P)$};

\node[right=2mm of f2, minimum width=2cm] (f3)
{$\longrightarrow\ X$};

\node[right=3mm of f3, minimum width=3.5cm, align=center] (f4)
{measurable,\\
$(f_i)_{u\in I}$ e.p.i.};

\node[f-functionBox, fit=(f1) (f2) (f3) (f4)] (f) {};

\draw[->] (varphi2)--(varphi2extension)
node[arrow_text, midway, left=4pt]
{$\exists$ extension};

\draw[->] (varphi2extension)--(f2extension-omega-product);
\draw[->] (omega)--(f2extension-omega-product);

\draw[->] (f2extension-omega-product)--(f2)
node[arrow_text, midway, left, align=center]
{$\exists$ extension\\
with Fubini property};

\draw[->, thick]
(varphi.west)--(f.west)
node[arrow_text, midway, left=3mm,align=center]
{apply \\Thm.\\
\ref{appendix.fubiniextension.theorem.existence}};

\draw[->, decorate, decoration={snake, amplitude=0.5mm, segment length=10mm}]
(f.east)--(varphi.east)
node[arrow_text, midway, right=3mm, align=center]
{$\P\circ f_i^{-1}$\\$=\varphi(u)$,\\ $\forall u\in I$};

\end{tikzpicture}
\caption{Visualization of Theorem~\ref{appendix.fubiniextension.theorem.existence}.}
\end{figure}

\noindent\textbf{Acknowledgements:}
S.~Grube and M.~R\"ockner are funded by the Deutsche Forschungsgemeinschaft (DFG, German Research Foundation) - Project-ID 317210226 - SFB~1283.
G.~Pang is partly funded by the US National Science Foundation Grant 2452849.

\bibliographystyle{abbrv}
\bibliography{bibliography}

\end{document}